\documentclass[12pt]{amsart}
\usepackage{amsfonts}
\usepackage{graphicx}
\usepackage{amscd}
\usepackage{amsmath}
\usepackage{amssymb}
\usepackage{amsfonts}
\usepackage{mathtools}
\usepackage[hmargin=1.2in]{geometry}

\newtheorem{theorem}{Theorem}[section]
\newtheorem{lemma}[theorem]{Lemma}
\newtheorem{corollary}[theorem]{Corollary}
\newtheorem{proposition}[theorem]{Proposition}

\theoremstyle{definition}
\newtheorem{definition}[theorem]{Definition}
\newtheorem{example}[theorem]{Example}
\newtheorem{claim}{Claim}
\theoremstyle{remark}
\newtheorem{remark}[theorem]{Remark}

\numberwithin{equation}{section}

\begin{document}
\title[General Stieltjes moment problems]{General Stieltjes moment problems for rapidly decreasing smooth functions}
\author{Ricardo Estrada}
\address{R. Estrada, Department of Mathematics\\
Louisiana State University\\
Baton Rouge, LA 70803\\
U.S.A.}
\email{restrada@math.lsu.edu}
\author{Jasson Vindas}
\address{J. Vindas\\
Department of Mathematics\\
Ghent University \\
Krijgslaan 281\\
B 9000 Gent\\
Belgium}
\email{jasson.vindas@UGent.be}
\thanks{J. Vindas gratefully acknowledges support from Ghent University, through the
BOF-grant number 01N01014}


\begin{abstract}
We give (necessary and sufficient) conditions over a sequence $\left\{
f_{n}\right\}  _{n=0}^{\infty}$ of functions under which every generalized
Stieltjes moment problem
\[
\int_{0}^{\infty} f_{n}(x)\phi(x)\mathrm{d} x=a_{n}, \ \ \ n\in\mathbb{N},
\]
has solutions $\phi\in\mathcal{S}(\mathbb{R})$ with $\operatorname*{supp}
\phi\subseteq[0,\infty)$. Furthermore, we consider more general problems of
this kind for measure or distribution sequences $\left\{ f_{n}\right\}
_{n=0}^{\infty}$. We also study vector moment problems with values in Fr\'{e}chet spaces and multidimensional moment problems.

\end{abstract}
\maketitle

\section{Introduction\label{Sectio: Introduction}}

The problem of moments, as its generalizations, is an important mathematical
problem which has attracted much attention for more than a century. It was
first raised and solved by Stieltjes for non-negative measures
\cite{shohat-tamarkin, Stieltjes}. Boas \cite{boas1939} and P\'{o}lya
\cite{polya1938} showed later that given an arbitrary sequence $\left\{
a_{n}\right\}  _{n=0}^{\infty}$ there is always a function of bounded
variation $F$ such that
\begin{equation}
\int_{0}^{\infty}x^{n}\mathrm{d}F(x)=a_{n}\,,\ \ \ n\in\mathbb{N}\,.
\label{momentseq1.1}
\end{equation}

A major improvement to this result was achieved by Dur\'{a}n, who was able to
show the existence of regular solutions to (\ref{momentseq1.1}). He proved in
\cite{duran1989} that every Stieltjes moment problem
\begin{equation}
\int_{0}^{\infty}x^{n}\phi(x)\,\mathrm{d}x=a_{n}\,,\ \ \ n\in\mathbb{N}\,,
\label{momentseq1.2}
\end{equation}
admits a solution $\phi\in\mathcal{S}(0,\infty)$, that is, a solution in the
Schwartz class of rapidly decreasing smooth functions $\mathcal{S}
(\mathbb{R})$ with $\operatorname*{supp}\phi\subseteq\lbrack0,\infty)$. The
corresponding generalization for the strong moment problem has been given in
\cite{duran-estrada1994}. Extensions of these results to vector-valued Stieltjes moment
problems are also well-known \cite{duran1992,estrada1998,galindo-sanz2001}. Chung,
Chung, and Kim, and more recently, Lastra and Sanz have considered moment
problems with solutions in Gelfand-Shilov classes
\cite{C-C-K2003,lastra-sanz2008,lastra-sanz2009}.

The problem that we are concerned with in this article is a general Stieltjes
moment problem in which we replace the sequence of monomials $\left\{
x^{n}\right\}  _{n=0}^{\infty}$ in \eqref{momentseq1.2} by a rather general sequence of functions (or
distributions) $\left\{  f_{n}\right\}  _{n=0}^{\infty}$. We are interested in
conditions over $\left\{  f_{n}\right\}  _{n=0}^{\infty}$ that ensure the
existence of solutions $\phi\in\mathcal{S}(0,\infty)$ to every infinite
system of equations
\begin{equation}
\int_{0}^{\infty}f_{n}(x)\phi(x)\,\mathrm{d}x=a_{n}\,,\ \ \ n\in\mathbb{N}\,,
\label{momentseq1.3}
\end{equation}
for a given arbitrary sequence $\left\{  a_{n}\right\}  _{n=0}^{\infty}$. In
particular, we shall show that the following conditions on the
asymptotic behavior of the primitives of the sequence suffice for the
solvability of every (\ref{momentseq1.3}):

\begin{theorem}
\label{momentsth1} Let $\left\{  f_{n}\right\}  _{n=0}^{\infty}$ be a sequence
of locally integrable functions on $[0,\infty)$ having at most polynomial
growth at infinity. Then, every generalized moment problem \eqref{momentseq1.3} has solutions $\phi$ in $\mathcal{S}(0,\infty)$ if the sequence $\left\{
f_{n}\right\}  _{n=0}^{\infty}$ fulfills:

\begin{itemize}
\item[$(i)$] $\int_{0}^{x} f_{n}(t)\mathrm{d}t=o\left(\int_{0}^{x}f_{n+1}
(t)\mathrm{d}t\right)$ as $x\to\infty$, $n\in\mathbb{N}$.

\item[$(ii)$] For every $\alpha>0$ there are $N=N_{\alpha}$ and $\sigma
=\sigma_{\alpha}>1$ such that
\begin{equation}
\inf_{a\in\lbrack1,\sigma]}\left\vert \int_{0}^{ax}f_{N}(t)\mathrm{d}
t\right\vert =\Omega(x^{\alpha})\,\ \ \ \mbox{as }x\rightarrow\infty\:.
\label{eqth1}%
\end{equation}

\end{itemize}
\end{theorem}

Theorem \ref{momentsth1} is a general version of Dur\'{a}n's theorem quoted
above. It also covers the case $f_{n}(x)=x^{\alpha_{n}}$ or generalized moment
problems such as
\begin{equation}
\int_{0}^{\infty}x^{\alpha_{n}}\sin\left(  \frac{1}{x^{\beta}}\right)
\phi(x)\mathrm{d}x=a_{n}\,,\ \ \ n\in\mathbb{N}\ \ \ (\beta>0)\:,
\label{momentseq1.6}
\end{equation}
where $-1<\Re e$\thinspace$\alpha_{0}<\Re e$\thinspace$\alpha_{1}<\dots<\Re
e$\thinspace$\alpha_{n}\rightarrow\infty$. The $\Omega$ in (\ref{eqth1})
stands for the Hardy-Littlewood symbol, namely, the negation of Landau's
little $o$ symbol: $f(x)=\Omega(g(x))$ (as $x\rightarrow\infty$) means that
there is a constant $C>0$ such that the inequality $|f(x)|\geq C|g(x)|$ holds
infinitely often for arbitrarily large values of $x$.

The assumptions on the sequence $\{f_{n}\}^{\infty}_{n=0}$ from Theorem
\ref{momentsth1} can be greatly relaxed. As we show, there is a corresponding
result that applies for sequences of functions that might not be even locally Lebesgue integrable.
In fact, we study in Section \ref{Moments for distribution sequence} the
general moment problem for distribution sequences $\left\{  f_{n}\right\}
_{n=0}^{\infty}$. We provide in Theorem \ref{momentsth2} a complete characterization of those distribution sequences for which all moment problems $a_{n}=\langle f_{n},\phi\rangle$, $n\in\mathbb{N}$, have solutions $\phi\in\mathcal{S}(0,\infty)$. The
notion of Ces\`{a}ro admissibility, introduced in Section
\ref{Moments for distribution sequence}, plays a key role in our criterion for
the existence of solutions in $\mathcal{S}(0,\infty)$ to generalized moment
problems. In Section \ref{function sequences}, we specialize our results to the case of
function sequences.

We consider in Section \ref{Section: measure} measure weighted moment problems of the form
\begin{equation}
\label{momentseq1.5}
\int_{0}^{\infty} x^{\alpha_{n}}\phi(x)\mathrm{d}F(x)=a_{n}\:, \quad n\in\mathbb{N}\:, 
\end{equation}
for a (fixed) non-negative measure $\mathrm{d}F$ and a sequence of real exponents $\{\alpha_{n}\}_{n=0}^{\infty}$. Theorem \ref{meas thm 1} below gives necessary and sufficient conditions for the solvability of every \eqref{momentseq1.5} in $\mathcal{S}(0,\infty)$. Interestingly, our general considerations apply to show the existence of solutions to moment
problems that could arise in quite different terms. For example, if $\left\{
\alpha_{n}\right\}  _{n=0}^{\infty}$ is an increasing sequence of real numbers tending to $\infty$, then,
as follows from our results, every discrete moment problem
\[
\sum_{p \text{ prime}}p^{\alpha_{n}}\phi(p)=a_{n}\,,\ \ \ n\in\mathbb{N}\,,
\]
admits solutions $\phi\in\mathcal{S}(0,\infty)$. 

Section \ref{Section: Vector moment problems} is devoted to vector moment problems with values in a Fr\'{e}chet space. Our analysis of the vector moment problem is based upon some results on the density of the set of solutions to moment problems, which will be obtained in Section \ref{Section: Density of the set of solutions of moment problems}. We conclude the article by studying moment problems in several variables in Section \ref{Section: Moment problems in several variables}.

\section{Preliminaries\label{Section: Preliminaries}}

We use the standard notation from distribution theory
\cite{campos,estrada-kanwal2002,vladimirovbook}. The space $\mathcal{S}_{+}^{\prime}$
denotes \cite{vladimirovbook} the space of all tempered distributions with
supports in the interval $[0,\infty)$. It can be canonically identified with
the dual space of $\mathcal{S}_{+}$, where
\[
\mathcal{S}_{+}=\left\{  \psi\in C^{\infty}[0,\infty):\,\,\psi=\varphi
_{|_{[0,\infty)}},\mbox{ for some }\varphi\in\mathcal{S}(\mathbb{R})\right\}
\,.
\]
Notice that $\mathcal{S}(0,\infty)$ is a closed subspace of $\mathcal{S}_{+}$,
where, as in the Introduction, $\mathcal{S}(0,\infty)$ consists of those
$\psi\in\mathcal{S}_{+}$ such that $\psi^{(m)}(0)=0$ for every
$m\in\mathbb{N}$.

We denote as $\mathcal{N}_{0}$ the annihilator of $\mathcal{S}(0,\infty)$ in
$\mathcal{S}^{\prime}_{+}$. It is then clear that $\mathcal{N}_{0}$ consists
of \emph{delta sums} at the origin, that is, finite linear combinations of the
Dirac delta $\delta$ and its derivatives.

We shall employ the notion of Ces\`{a}ro behavior of distributions, introduced
in \cite{estradaCesaro} (see also
\cite{estrada-kanwal2002,estrada-vindasTauber2nd,P-S-V}). We start with
primitives of distributions \cite{vladimirovbook}. Given $f\in\mathcal{S}
_{+}^{\prime}$ and $m\in\mathbb{N}$, we denote as $f^{(-m)}$ the $m$-primitive
of $f$ that satisfies $f^{(-m)}\in\mathcal{S}_{+}^{\prime}$. It can be
expressed as the convolution \cite{vladimirovbook}
\[
f^{(-m)}=f\ast\frac{x_{+}^{m-1}}{(m-1)!}\,.
\]
The Ces\`{a}ro order growth symbols for distributions are defined as follows.
Let $\alpha\in\mathbb{R}\setminus\{-1,-2,\dots\}$ and $m\in\mathbb{N}$, for
$f\in\mathcal{S}_{+}^{\prime}$, we write
\[
f(x)=O(x^{\alpha})\ \ \ (\mathrm{C},m)\,,\ \ \ x\rightarrow\infty\,,
\]
if $f^{(-m)}$ is a regular distribution (locally Lebesgue integrable) for
large arguments and there is a polynomial $P$ of degree at most $m-1$ such
that
\begin{equation}
f^{(-m)}(x)=P(x)+O(x^{\alpha+m})\ \ \ \mbox{as }x\rightarrow\infty\,.
\label{momentseq6}
\end{equation}
Observe that if $\alpha>-1$ the polynomial in (\ref{momentseq6}) is then
irrelevant. The little $o$ symbol is defined in a similar fashion. When $f$ is
locally Lebesgue integrable, the relation (\ref{momentseq6}) reads as
\[
\frac{1}{x}\int_{0}^{x}f(t)\left(  1-\frac{t}{x}\right)  ^{m-1}\,\mathrm{d}
t=\frac{Q(x)}{x^{m}}+O(x^{\alpha})\,,
\]
for some polynomial $Q$ of degree at most $m-1$. We also introduce the
Hardy-Littlewood symbol $\Omega(x^{\alpha})$ in the Ces\`{a}ro sense. Thus, we
define
\[
f(x)=\Omega(x^{\alpha})\ \ \ (\mathrm{C},m)\,,\ \ \ x\rightarrow\infty\,,
\]
as the negation of $f(x)=o(x^{\alpha})$ $\ (\mathrm{C},m)$, $x\rightarrow
\infty$; in particular, if $\alpha>-1$ and $f^{(-m)}(x)$ is a function for
large $x$, it just means that $f^{(-m)}(x)=\Omega\left(  x^{\alpha+m}\right)
$. Analogous definitions apply as $x\rightarrow0^{+}$.

If we do not have to emphasize the role of $m$ in a Ces\`{a}ro order relation
we simply write $(\mathrm{C})$, which stands for $(\mathrm{C},m)$ for some
$m$. The growth order symbols can be used to define asymptotic relations and
distributional evaluations in the Ces\`{a}ro sense, see
\cite{estrada-kanwal2002} for details. 

\section{General moment problems for distribution
sequences\label{Moments for distribution sequence}}

We study in this section the following general moment problem. Let $\left\{
f_{n}\right\}  _{n=0}^{\infty}\subset\mathcal{S}_{+}^{\prime}$ be a sequence
of distributions. We seek conditions over $\left\{  f_{n}\right\}
_{n=0}^{\infty}$ such that every generalized moment problem, for a given
arbitrary sequence $\left\{  a_{n}\right\}  _{n=0}^{\infty}$,

\begin{equation}
\left\langle f_{n},\phi\right\rangle =a_{n}\,,\ \ \ n\in\mathbb{N}\,,
\label{momentseq7}
\end{equation}
admits a solution $\phi\in\mathcal{S}(0,\infty)$.

We start with a natural condition over $\left\{
f_{n}\right\}  _{n=0}^{\infty}$. First, notice that if (\ref{momentseq7}) is solvable for arbitrary sequences we must necessarily have
\begin{equation}
\tag{\textbf{P1}'}
\label{P1'}
\mbox{the distibutions }f_{0},f_{1},f_{2}
,\dots,f_{n},\dots\mbox{ are linearly independent.}
\end{equation}
Since we are interested in solutions to (\ref{momentseq7}) in $\mathcal{S}
(0,\infty)$, one should assume that none of the $f_{n}$ is a delta sum at the
origin; therefore, we assume 
the following relation between the sequence and the annihilator of
$\mathcal{S}(0,\infty)$:
\begin{equation}
\tag{\textbf{P1}}
\label{P1} \mbox{No element of }\mathcal{N}_{0}\mbox{ is a
linear combination of }f_{0},f_{1},f_{2},\dots,f_{n},\dots\,.
\end{equation}
Naturally \eqref{P1} implies \eqref{P1'} since one
considers the zero distribution is a delta sum with zero coefficients.

A key property in our criterion for the existence of solutions to
(\ref{momentseq7}) is the notion of Ces\`{a}ro admissibility, defined as follows.
We shall say that the sequence $\left\{  f_{n}\right\}  _{n=0}^{\infty}$ is
\emph{Ces\`{a}ro admissible} \emph{or that it satisfies the property} \eqref{P2} if: There is an increasing sequence of integers $\left\{
m_{j}\right\}  _{j=0}^{\infty}$ such that for every $j\in\mathbb{N}$ and every
$\alpha>0$ ($\alpha\notin\mathbb{Z}$) there exists $\nu=\nu_{j,\alpha}
\in\mathbb{N}$ such that if $N\geq\nu$
\begin{equation}
 \tag{\textbf{P2}}
 \label{P2}
\begin{drcases} 
&
\sum_{n=0}^{N}b_{n}f_{n}(x)=O(x^{\alpha})\ \ (\mathrm{C},m_{j}
)\:,\ \ x\rightarrow\infty\, ,
\\
&
\left( \sum_{n=0}^{N}b_{n}{f_{n}^{(-m_{j})}}\right)_{|(0,\infty)} \in C(0,\infty)\:, \mbox{ and}
  \\
  &
  \sum_{n=0}^{N}b_{n}f_{n}(x)=O(x^{-\alpha})\ \ (\mathrm{C},m_{j}
)\:,\ \ x\rightarrow0^{+}
\end{drcases}
  \Longrightarrow b_{\nu}=b_{\nu+1}=\dots=b_{N}=0\,.
\end{equation}
Note that \eqref{P2} is equivalent to: For each $m\in\mathbb{N}$ and $\alpha\in\mathbb{R}
_{+}\setminus{\mathbb{N}}$ there exits $\nu=\nu_{\alpha,m}\in\mathbb{N}$ such
that if $N\geq\nu$ and $b_{N}\neq0$, then 
$$
\sum_{n=0}^{N}b_{n}f_{n}(x)=\Omega(x^{\alpha}) \quad \ (\mathrm{C},m)\:, \quad x\rightarrow\infty\:, \quad \mbox{or} 
$$
$$
\left(\sum_{n=0}^{N}b_{n}f_{n}^{(-m)}\right)_{|(0,\infty)}\notin
C(0,\infty)\:, \quad \mbox{or}
$$

$$
\sum_{n=0}^{N}b_{n}f_{n}(x)=\Omega(x^{-\alpha}) \quad \ (\mathrm{C},m)\:, \quad x\rightarrow0^{+}\: ,
$$
holds. For example, the sequence of monomials $f_{n}(x)=x^{n}$ is Ces\`{a}ro
admissible, but so are $f_{n}(x)=\operatorname*{Pf}(x^{-n})$ and $f_{n}(x)=\delta^{(n)}(x-1)$, $n\in\mathbb{N}$.

The ensuing theorem is the main result of this section.

\begin{theorem}
\label{momentsth2} Every generalized moment problem $(\ref{momentseq7})$ has a
solution $\phi\in\mathcal{S}(0,\infty)$ if and only if the distribution
sequence $\left\{  f_{n}\right\}  _{n=0}^{\infty}$ satisfies the properties \eqref{P1} and \eqref{P2}.
\end{theorem}

In the proof of Theorem \ref{momentsth2}, we shall employ a result of Silva.

\begin{lemma}
[Silva \cite{silva1955}]\label{momentsl1} Let $E$ be a Silva space, namely,
$E$ is the inductive limit of an increasing sequence of Banach spaces
$\left\{  E_{n}\right\}  _{n=0}^{\infty}$ where each inclusion mapping
$E_{n}\to E_{n+1}$ is compact. Then, a linear subspace $X\subset E$ is closed
if and only if $X\cap E_{n}$ is closed in each $E_{n}$.
\end{lemma}

We point out that the class of Silva spaces is precisely that of (DFS)-spaces
(strong duals of Fr\'{e}chet-Schwartz spaces).

\begin{proof}
[Proof of Theorem \ref{momentsth2}]Let $\mathbb{C}[[\xi]]$ be the space of formal
power series in one indeterminate with the topology of convergence in each
coefficient. Its dual is the space of polynomials in one-variable, denoted
here as $\mathcal{P}$. That every generalized moment problem (\ref{momentseq7}) has a solution $\phi\in\mathcal{S}(0,\infty)$ is equivalent to the
surjectivity of the continuous linear mapping
\[
\Lambda:\mathcal{S}(0,\infty)\rightarrow\mathbb{C}[[\xi]]\,,
\]
given by
\[
\Lambda(\psi)=\sum_{n=0}^{\infty}\mu_{n}(\psi)\xi^{n}\,,\ \ \ \mbox{where }\mu
_{n}(\psi)=\left\langle f_{n},\psi\right\rangle\,.
\]
By the well-known criterion for surjectivity of continuous linear mappings
between Fr\'{e}chet spaces \cite[Thm. 37.2, p. 382]{treves}, the mapping $\Lambda$ is surjective
if and only if its transpose
\[
\Lambda^{\top}:\mathcal{P}\rightarrow \mathcal{S}'(0,\infty)=\mathcal{S}_{+}^{\prime}/\mathcal{N}
_{0}\,,
\]
is injective and has weakly$^{\ast}$ closed range. Since $\mathcal{S}
_{+}^{\prime}/\mathcal{N}_{0}$ is Montel, it is reflexive, and there is
therefore no distinction between weakly$^{\ast}$ closedness or strong
closedness for its linear subspaces. The transpose of $\Lambda$ is easily seen
to be given by
\[
\Lambda^{\top}\left(  \sum_{n=0}^{N}b_{n}\xi^{n}\right)  =\sum_{n=0}^{N}
b_{n}f_{n}\,_{|\mathcal{S}(0,\infty)}\,.
\]
It is then obvious that $\Lambda^{\top}$ is injective if and only if \eqref{P1} holds. Write $\pi$ for the quotient mapping $\pi:\mathcal{S}_{+}^{\prime}\rightarrow\mathcal{S}_{+}^{\prime
}/\mathcal{N}_{0}$.

We first show the sufficience of \eqref{P1} and \eqref{P2} for $\Lambda^{\top}$ to have
closed range. Let $\mathcal{M}$ be the linear span of the $f_{0},f_{1}
,\dots,f_{n},\dots$ in $\mathcal{S}_{+}^{\prime}$. Furthermore, the range of
$\Lambda^{\top}$ is $\pi(\mathcal{M})$. We now use the fact that
$\mathcal{S}_{+}^{\prime}$ is a Silva space. For its defining inductive
sequence, we use the choice as in \cite{vladimirovbook}. For each
$n\in\mathbb{N}$, let $\mathcal{S}_{+,n}$ be the completion of $\mathcal{S}
_{+}$ in the norm
\[
||\psi||_{n}=\sup_{x\in\lbrack0,\infty),\,j\leq n}(1+x^{2})^{n/2}|\psi
^{(j)}(x)|\,.
\]
Clearly, $\mathcal{S}_{+,n}$ consists of those functions $\psi\in C^{n}[0,\infty)$ such that $\lim_{x\to\infty} x^{n}\psi^{(j)}(x)=0$ for $0\leq j\leq n$.
Then, the injection $\mathcal{S}_{+,n}^{\prime}\rightarrow\mathcal{S}_{+,n+1}^{\prime}$ is
compact and $\mathcal{S}_{+}^{\prime}=\bigcup_{n=0}^{\infty}\mathcal{S}
_{+,n}^{\prime}$, topologically. Moreover, denoting $\mathcal{S}_{0,n}^{\prime}
=\pi(\mathcal{S}_{+,n}^{\prime})$, we have that $\mathcal{S}^{\prime}
(0,\infty)=\bigcup_{n=0}^{\infty}\mathcal{S}_{0,n}^{\prime}$ topologically and each
$\mathcal{S}_{0,n}^{\prime}\rightarrow\mathcal{S}_{0,n+1}^{\prime}$ is also
compact. By Lemma \ref{momentsl1}, $\pi(\mathcal{M})$ is closed if and only if
$\pi(\mathcal{M})\cap\mathcal{S}_{0,n}^{\prime}$ is closed in $\mathcal{S}
_{0,n}^{\prime}$ for each $n$. Since finite dimensional subspaces are always
closed, the latter will be a consequence of the following claim:

\begin{claim}Ces\`{a}ro admissibility implies that $\pi\left(
\mathcal{M}\right)  \cap\mathcal{S}_{0,n}^{\prime}$ is finite dimensional for each
$n\in\mathbb{N}.$
\end{claim}

Indeed, for each $k\in\mathbb{N}$, set
\[
\mathcal{X}_{k}=\left\{  \sum_{n=0}^{k}b_{n}f_{n}:\,b_{n}\in\mathbb{C}
\right\}  \oplus\mathcal{N}_{0}\subset\mathcal{M}\oplus\mathcal{N}_{0}\,.
\]
Let $\left\{  m_{j}\right\}  _{j=0}^{\infty}$ be the sequence from \eqref{P2}. We show
that for $p=m_{j}-2>0$, there is $k=k_{j}$ such that $(\mathcal{M}
\oplus\mathcal{N}_{0})\cap\mathcal{S}_{+,p}^{\prime}\subseteq\mathcal{X}_{k}$,
whence the claim would follow. Suppose that $g=\sum_{n=0}^{N}(b_{n}f_{n}
+c_{n}\delta^{(n)})\in\mathcal{S}_{+,p}^{\prime}$. Find $\nu$ such that \eqref{P2}
holds for $j$ and $\alpha=2p+3/2$. Observe that $\varphi_{x}(u)=(x-u)_{+}
^{p+1}\in\mathcal{S}_{+,p}$. Thus,
\[
g^{(-m_{j})}(x)=g^{(-p-2)}(x)=g\ast\frac{x_{+}^{p+1}}{(p+1)!}=\frac{1}
{(p+1)!}\left\langle g(u),\varphi_{x}(u)\right\rangle
\]
is a continuous function (on the whole $[0,\infty)$) and
\begin{align*}
|g^{(-p-2)}(x)|  &  \leq\frac{||g||_{\mathcal{S}_{+,p}^{\prime}}||\varphi
_{x}||_{p}}{(p+1)!}=O\left(  x^{p+1}\sup_{u\in\lbrack0,x),j\leq p}(|u|^{p}+1)\left(
1-u/x\right)  ^{p+1-j}\right) \\
&  =O(x^{2p+1})\,.
\end{align*}
Hence, $\sum_{n=0}^{N}b_{n}f_{n}(x)=O(x^{\alpha})$ $(\mathrm{C},m_{j})$,
$x\rightarrow\infty$, $\sum_{n=0}^{N}b_{n}f_{n}^{(-m_{j})}(x)$ is continuous
for $x\in(0,\infty)$, and $\sum_{n=0}^{N}b_{n}f_{n}(x)=O(x^{-\alpha})$
$(\mathrm{C},m_{j})$, $x\rightarrow0^{+}$. We obtain from \eqref{P2} that $b_{n}=0$
for $n\geq\nu$. Thus, $g\in\mathcal{X}_{\nu-1}$. The claim has been established.

Conversely, assume that $\Lambda^{\top}$ is injective and has weakly$^{\ast}$
closed range. As already pointed out, the range $\Lambda^{\top}(\mathcal{P})$
is thus strongly closed because $\mathcal{S}_{+}^{\prime}/\mathcal{N}_{0}$ is
reflexive (in fact a (DFS)-space). We have already noticed that \eqref{P1} must
necessarily hold. To show \eqref{P2}, we first establish that $\Lambda^{\top}$ is an
isomorphism into its image. Pt\'{a}k's theory \cite{kotheII,rr} applies to
show that $\Lambda^{\top}$\thinspace$\mathcal{P}\rightarrow\Lambda^{\top
}(\mathcal{P})$ is open if we verify that $\mathcal{P}$ is fully complete
($B$-complete in the sense of Pt\'{a}k) and that $\Lambda^{\top}(\mathcal{P})$
is barreled. It is well known \cite[p. 123]{rr} that the strong dual of a
reflexive Fr\'{e}chet space is fully complete, so $\mathcal{P}$, as a (DFS)-space, is fully complete. Now, a closed subspace of a (DFS)-space must itself be
a (DFS)-space. Since $\mathcal{S}_{+}^{\prime}/\mathcal{N}_{0}$ is a (DFS)-space,
we obtain that $\Lambda^{\top}(\mathcal{P})$ is also a (DFS)-space and hence barreled.

Suppose now that \eqref{P2} were false. Then, there are $j$ and $\alpha>0$
such that $g_{n}(x)$\thinspace$=\sum_{\nu=0}^{n}b_{\nu,n}f_{\nu}
(x)=O(x^{\alpha})$ $(\mathrm{C},j)$, $x\rightarrow\infty$, $g_{n}
(x)=O(x^{-\alpha})$ $(\mathrm{C},j)$, $x\rightarrow0^{+}$, and $g_{n}
^{(-j)}(x)\in C(0,\infty),$ with $b_{n,n}\neq0$ for infinitely many indices
$n$. If $p\geq\max\{\alpha+j+2\}$, we conclude that there is an increasing
sequence of indices $n_{0}<n_{1}<\dots<n_{k}<\dots$ such that $\{\pi(g_{n_{k}
})\}_{k=0}^{\infty}\subset\mathcal{S}_{0,p}^{\prime}$. Let $h_{n_{k}}
=\sum_{\nu=0}^{n_{k}}(b_{\nu,n_{k}}/\Vert\pi(g_{n_{k}})\Vert_{\mathcal{S}
_{0,p}^{\prime}})\pi(f_{\nu})=$\thinspace$\sum_{\nu=0}^{n_{k}}a_{\nu,n_{k}}
\pi(f_{\nu})$ with $a_{n_{k},n_{k}}\neq0$. Then $\{h_{n_{k}}\}_{k=0}^{\infty}$
is bounded in $\mathcal{S}_{0}^{\prime}$ because of the continuity of the inclusion mapping
$\mathcal{S}_{0,p}^{\prime}\rightarrow\mathcal{S}^{\prime}(0,\infty)$. We deduce
then that $\{h_{n_{k}}\}_{k=0}^{\infty}$ is a bounded set of $\Lambda
^{T}(\mathcal{P})$ and since $\Lambda^{\top}$ is open, we should have that the
set of polynomials $\{\sum_{\nu=0}^{n_{k}}a_{\nu,n_{k}}\xi^{\nu}
\}_{k=0}^{\infty}$ is bounded in $\mathcal{P}$ as well. But the latter can
only hold if there is $k_{0}$ such that $a_{\nu,n_{k}}=0$ for all indices
$n_{k}\geq n_{k_{0}}$ and $\nu\geq n_{k_{0}}$, which produces a contradiction.
This completes the proof of the theorem.
\end{proof}

Let us discuss a simple example to illustrate Theorem \ref{momentsth2}.

\begin{example}[The generalized Borel problem]
\label{momentsex1} Let $\{k_{n}\}_{n=0}^{\infty}$ be sequence of natural numbers with $k_{n}\to\infty$ and let $\{x_{n}\}_{n=0}^{\infty}$ be a sequence of positive real numbers such that all pairs $(k_{n},x_{n})$ are distinct. The distribution sequence $\{f_{n}\}_{n=0}^{\infty}$ given by $f_{n}(x)=\delta^{(k_{n})}(x-x_{n})$ satisfies \eqref{P1} and \eqref{P2}. Consequently, Theorem \ref{momentsth2} gives that 
\begin{equation}
\label{momentsexeq1}
\phi^{(k_n)}(x_{n})=a_{n}, \quad n\in\mathbb{N},
\end{equation}
is always solvable in $\mathcal{S}(0,\infty)$. If $\{x_{n}\}_{n=0}^{\infty}$ stays on a fixed compact subset of $(0,\infty)$, multiplying by a cut-off function, we can in fact find solutions to (\ref{momentsexeq1}) that belong to $\mathcal{D}(0,\infty)$. In particular, choosing $x_{n}=x_{0}$ to be constant and $k_{n}=n$ the sequence of all natural numbers, we recover the well known fact that 
\begin{equation}
\label{momentseqBorel}
\phi^{(n)}(x_{0})=a_{n}\,,\ \ \ n\in\mathbb{N}\,. 
\end{equation}
has solution $\phi\in \mathcal{D}(\mathbb{R})$. Performing a translation, one easily sees that one may take an arbitrary
$x_{0}\in\mathbb{R}$ in (\ref{momentseqBorel}), that is, every classical Borel problem has
a solution $\phi\in\mathcal{D}(\mathbb{R})$. 

\smallskip

\end{example}

Theorem \ref{momentsth2} can be generalized to distribution two-sided
sequences $\{f_{n}\}_{n\in\mathbb{Z}}$. In fact, consider
\begin{equation}
\left\langle f_{n},\phi\right\rangle =a_{n}\,,\ \ \ n\in\mathbb{Z}\,.
\label{momentseq3.2}
\end{equation}
If we set for $n\in\mathbb{N}$
\[
c_{n}\,=
\begin{cases}
a_{n/2} & \mbox{if }n\mbox{ is  even}\\
a_{-(n+1)/2} & \mbox{if }n\mbox{ is  odd}
\end{cases}
\ \ \ \mbox{ and }\ \ \ g_{n}\,=
\begin{cases}
f_{n/2} & \mbox{if }n\mbox{ is  even}\\
f_{-(n+1)/2} & \mbox{if }n\mbox{ is  odd}\,,
\end{cases}
\]
then the moment problem (\ref{momentseq3.2}) is equivalent to $\left\langle
g_{n},\phi\right\rangle =c_{n}$, $n\in\mathbb{N}$, while that \eqref{P1} and \eqref{P2} hold for $\{g_{n}\}_{n=0}^{\infty}$ translates into:

\begin{enumerate}

\item[(\textbf{P1}*)] no element of $\mathcal{N}_{0}$ is a linear combination of
$\dots, f_{-2},f_{-1},f_{0},f_{1},f_{2},\dots$, and

\item[(\textbf{P2}*)] there is an increasing sequence of integers $\left\{
m_{j}\right\}  _{j=0}^{\infty}$ such that for each $j\in\mathbb{N}$ and each
$\alpha>0$ there exists $\nu=\nu_{j,\alpha}\in\mathbb{N}$ such that if
$N\geq\nu$: $\sum_{n=-N}^{N}b_{n} f_{n}(x)=O(x^{\alpha})$ $\ (\mathrm{C}
,m_{j})$, $x\to\infty$, $\sum_{n=-N}^{N}b_{n} f_{n}^{(-m_{j})}(x)$ is
continuous for $x\in(0,\infty)$ and $\sum_{n=-N}^{N}b_{n} f_{n}
(x)=O(x^{-\alpha})$ $\ (\mathrm{C},m_{j})$, $x\to0^{+}$, imply $b_{n}=0$ for
every $|n|\geq\nu$.
\end{enumerate}

We thus obtain,

\begin{corollary}
\label{momentsc1} Let $\{f_{n}\}_{n\in\mathbb{Z}}\subset\mathcal{S}
_{+}^{\prime}$. Every $(\ref{momentseq3.2})$ admits a solution $\phi
\in\mathcal{S}(0,\infty)$ if and only if the conditions $(\bf{P1}$*$)$ and $(\bf{P2}$*$)$ are satisfied.
\end{corollary}

We end this section with an example.

\begin{example}
\label{momentsex2} The so called strong moment problem,
$$
a_{n}=\int_{0}^{\infty} x^{n}\phi(x)\mathrm{d}x, \quad n\in\mathbb{Z}\:,
$$
was studied and solved in \cite{duran-estrada1994} for $\phi\in\mathcal{S}(0,\infty)$. Using Corollary \ref{momentsc1}, we can strengthen the main result of \cite{duran-estrada1994} as follows. Let $\{\alpha_{n}\}_{n\in\mathbb{Z}}$ be a two-sided sequence. Then, every strong moment problem
$$
a_{n}=\int_{0}^{\infty} x^{\alpha_n}\phi(x)\mathrm{d}x, \quad n\in\mathbb{Z},
$$
is solvable in $\mathcal{S}(0,\infty)$ if and only if all elements of the sequence are distinct and $|\Re e\: \alpha_{n}|\to\infty$ as $|n|\to\infty$.
\end{example}
\begin{remark}
\label{momentsrk1} The simple reduction explained above for moment problems with two-sided sequences also applies to all results from the next sections, we will therefore omit any comment concerning two-sided sequences in the sequel. 

\end{remark}

\section{Moment problems with function sequences}\label{function sequences}

We now focus our attention on function sequences. Throughout this section $\{f_{n}\}_{n=0}^{\infty}\subset\mathcal{S}%
_{+}^{\prime}$ stands for a sequence such that each $f_{n}$ is a non-identically zero locally integrable\footnote{All results from this section are valid for locally
distributionally integrable functions in the sense of the authors
\cite{estrada-vindasGIntegral}, and in particular for sequences of locally
Denjoy-Perron-Henstock or Lebesgue integrable functions on $(0,\infty)$.} 
function with continuous primitives\footnote{If $f_{n}$ is locally Denjoy-Perron-Henstock or Lebesgue integrable, its primitive is of course continuous. The primitives of distributional integrable functions are \L ojasiewicz functions but in general they may be discontinuous \cite{estrada-vindasGIntegral}, whence this assumption.} on $(0,\infty)$. Note that we allow $f_{n}$ not being
integrable near $x=0$. The distributional evaluation $\langle f_{n}
,\varphi\rangle$ for $\phi\in\mathcal{S}(0,\infty)$ can be always written
\cite{vindas-estrada2010} as a Ces\`{a}ro integral and thus we may rewrite in
this case the moment problem (\ref{momentseq7}) as
\begin{equation}
\int_{0}^{\infty}\phi(x)f_{n}(x)\mathrm{d}x=a_{n}\ \ \ (\mathrm{C}
),\ \ \ n\in\mathbb{N}\,, \label{momentseq4.1}
\end{equation}
Naturally, if each $f_{n}$ has at most polynomial growth, then the Ces\`{a}ro
integrals in (\ref{momentseq4.1}) can be replaced by ordinary integrals.

Theorem \ref{momentsth1} gives already a complete characterization of those
$\{f_{n}\}_{n=0}^{\infty}$ for which every moment problem is solvable in
$\mathcal{S}(0,\infty)$. Note that (\ref{P1'}) for $\{f_{n}
\}_{n=0}^{\infty}$ becomes equivalent to \eqref{P1}. On the other
hand, \eqref{P2} forces the linear combinations of the primitives of the sequence to have the ensuing
growth property: For each $m\in\mathbb{N}$ and $\alpha\in\mathbb{R}
_{+}\setminus{\mathbb{N}}$ there exits $\nu=\nu_{\alpha,m}\in\mathbb{N}$ such
that if $N\geq\nu$ and $b_{N}\neq0$, then
\begin{equation}
\sum_{n=0}^{N}b_{n}f^{(-m)}_{n}(x)=\Omega(x^{\alpha})\ \ \ \mbox{as }x\rightarrow
\infty\,\ \ \mbox{or}\ \ \sum_{n=0}^{N}b_{n}f_{n}^{(-m)}(x)=\Omega(x^{-\alpha
})\ \ \ \mbox{as }x\rightarrow0^{+}\,. \label{momentseq4.2}
\end{equation}
Since all primitives of $f_{n}$ are continuous on $(0,\infty)$, the latter property is actually equivalent to condition \eqref{P2}.

The next theorem tells that if for a fixed $m$ one slightly strengthens these
two conditions, then one obtains, together with linear independence, a useful
criterion for the solvability of arbitrary moment problems (\ref{momentseq4.1}).

\begin{theorem}
\label{momentsth3} Let $m\geq1$. Suppose that $\{f_{n}\}_{n=0}^{\infty}$
satisfies \eqref{P1'} and has the following property:

\begin{enumerate}
\item[(\textbf{P3})] For any given $\alpha>0$ there is $\nu
=\nu_{\alpha}\in\mathbb{N}$ such that if $N\geq\nu$ and $b_{N}\neq0$, then one can find $\sigma=\sigma_{N}>1$ such that
\begin{align}
&  \inf_{a\in\lbrack1,\sigma]}\left\vert \sum_{n=0}^{N}b_{n}f_{n}
^{(-m)}(ax)\right\vert =\Omega(x^{\alpha})\ \ \ \mbox{as }\,x\rightarrow
\infty\,,\label{momentseq4.3}\\
&  \ \ \mbox{or}\ \ \nonumber\\
&  \inf_{a\in\lbrack1,\sigma]}\left\vert \sum_{n=0}^{N}b_{n}f_{n}
^{(-m)}(ax)\right\vert =\Omega(x^{-\alpha})\ \ \ \mbox{as }\,x\rightarrow
0^{+}\,.\nonumber
\end{align}

\end{enumerate}
Then every generalized moment problem \eqref{momentseq4.1} is solvable in
$\mathcal{S}(0,\infty)$. If additionally each $f_{n}$ is a Darboux function\footnote{Every
continuous function is of course a Darboux function. More generally, if each
$f_{n}$ is a \emph{\L ojasiewicz function} it must have the Darboux property
\cite{estrada-vindasGIntegral, lojasiewicz}.} (i.e., has the intermediate
value property), then \eqref{momentseq4.3} might be replaced by
\begin{align}
&  \inf_{a\in\lbrack1,\sigma]}\left\vert \sum_{n=0}^{N}b_{n}f_{n}
(ax)\right\vert =\Omega(x^{\alpha})\ \ \ \mbox{as }\,x\rightarrow
\infty\,,\ \label{momentseq4.4}\\
&  \ \ \ \mbox{or}\nonumber\\
&  \inf_{a\in\lbrack1,\sigma]}\left\vert \sum_{n=0}^{N}b_{n}f_{n}
(ax)\right\vert =\Omega(x^{-\alpha})\ \ \ \mbox{as }\,x\rightarrow
0^{+}\,.\nonumber
\end{align}

\end{theorem}

\begin{proof}
In view of Theorem \ref{momentsth1}, we only need to show that
(\textbf{P3}) ensures the validity of (\ref{P2}). Suppose that
(\textbf{P3}) is satisfied but (\ref{P2}) does not hold. Then, one
can find $\alpha>j\geq1$ ($\alpha\notin\mathbb{N}$) and a sequence $g_{k}
=\sum_{n=0}^{N_{k}}b_{n,N_{k}}f_{n}$ such that $b_{N_{k},N_{k}}\neq0$ and
either $g_{k}^{(-j-m)}(x)=o(x^{\alpha+j})$ as $x\rightarrow\infty$ or
$g_{k}^{(-j-m)}(x)=o(x^{j-\alpha})$ as $x\rightarrow0^{+}$. The latter implies
that, for each $\phi\in\mathcal{S}(0,\infty)$, 
\begin{equation}
\lim_{\lambda\rightarrow\infty}\frac{1}{\lambda^{\alpha}}\int_{0}^{\infty
}g_{k}^{(-m)}(\lambda x)\phi(x)\mathrm{d}x=0\ \ \mbox{or}\ \ \lim
_{\lambda\rightarrow0^{+}}\lambda^{\alpha}\int_{0}^{\infty}g_{k}
^{(-m)}(\lambda x)\phi(x)\mathrm{d}x=0\,. \label{momentseq4.5}
\end{equation}

Let $\nu$ be the integer corresponding to $\alpha$ in (\textbf{P3}).
Fix $k$ and $\sigma>1$. Taking a nonnegative test function $\phi$ with
$\operatorname*{supp}\phi\subset\lbrack1,\sigma]$ and $\int_{1}^{\sigma}
\phi(x)\mathrm{d}x=1$ in (\ref{momentseq4.5}) and applying the mean-value
theorem \cite[Sect. 11]{estrada-vindasGIntegral}, we obtain that for each
$\lambda$ one can find $a_{\lambda}\in(1,\sigma)$ such that either
\begin{equation}
\lim_{\lambda\rightarrow\infty}\frac{g_{k}^{(-m)}(\lambda
a_{\lambda})}{\lambda^{\alpha}}=0\quad \mbox{or}\quad \lim_{\lambda\rightarrow\infty
}\frac{g_{k}^{(-m)}(a_{\lambda}/\lambda)}{\lambda^{\alpha}}=0\,.
\label{momentseq4.6}
\end{equation}
On the other hand, (\textbf{P3}) tells us that if $N_{k}>\nu$ there are $\sigma>1$, a sequence $\lambda_{n}\rightarrow\infty$, and $C>0$ such that
either $|g_{k}^{(-m)}(\lambda_{n}a)|\geq C\lambda_{n}^{\alpha}$ or
$|g_{k}^{(-m)}(a/\lambda_{n})|\geq C\lambda_{n}^{\alpha}$ for all $a\in
\lbrack1,\sigma]$ and $n\in\mathbb{N}$, contradicting (\ref{momentseq4.6}).

Note that integration by parts in (\ref{momentseq4.5}) leads to
\[
\mbox{either }\ \ \ \lim_{\lambda\rightarrow\infty}\frac{1}{\lambda^{\alpha
-m}}\int_{0}^{\infty}g_{k}(\lambda x)\phi(x)\mathrm{d}x=0\ \ \mbox{or}\ \ \lim
_{\lambda\rightarrow0^{+}}\lambda^{\alpha+m}\int_{0}^{\infty}g_{k}(\lambda
x)\phi(x)\mathrm{d}x=0\,.
\]
If we assume that each $f_{n}$ is a Darboux function, then so is each $g_{k}$.
The mean-value theorem still applies\footnote{See the proof of \cite[Prop.
11.1]{estrada-vindasGIntegral}.} to prove that assuming that \eqref{P2} fails is
contradictory with (\ref{momentseq4.4}). The result has been established.
\end{proof}

It should be noticed that Theorem \ref{momentsth1}, stated in the
Introduction, immediately follows by taking $m=1$ in Theorem \ref{momentsth3}.

Next, we are interested in weigthed moment problems of the form
\begin{equation}
\label{momentseq4.7}\int_{0}^{\infty} x^{\alpha_{n}}f(x)\phi(x)\mathrm{d}
x=a_{n} \ \ \ (\mathrm{C}), \ \ \ n\in\mathbb{N},
\end{equation}
where $f\in\mathcal{S}^{\prime}_{+}$ is a (non-identically zero) locally integrable function (with continuous primitive) on $(0,\infty)$. We point out that (\ref{P1}) holds if and only if
 $\{\alpha_{n}\}_{n=0}^{\infty}$ consists of distinct
complex numbers. We will prove that (\ref{P2}) forces the
sequence $\{\alpha_{n}\}_{n=0}^{\infty}$ to stay off vertical strips except
for finitely many terms.

\begin{theorem}
\label{momentsth4} Let $\{\alpha_{n}\}_{n=0}^{\infty}$ be a sequence of distinct complex numbers with the property there is $n_{0}$ such that $\Re e\:\alpha_{n}\neq\Re e\: \alpha_{m}$ for all distinct $n,m>n_{0}$. Then, every moment problem \eqref{momentseq4.7} is solvable in
$\mathcal{S}(0,\infty)$ if and only if 
\begin{equation}
\label{momentseq4.8}
\lim_{n\to\infty}|\Re e\: \alpha_{n}|=\infty
\end{equation}
and
\begin{enumerate}
\item[(I)] if $\displaystyle\lim_{n\rightarrow\infty}\Re e\:\alpha_{n}=\infty$, then for each $m\in\mathbb{N}$ the $m$-primitive of $f$
satisfies
\begin{equation}
-\infty<\limsup_{x\rightarrow\infty}\frac{\log|f^{(-m)}(x)-P(x)|}{\log x}\,,
\label{momentseq4.9}
\end{equation}
for each polynomial $P$ of degree at most $m-1$;

\item[(II)] if $\displaystyle\lim_{n\rightarrow\infty}\Re e\:\alpha_{n}=-\infty,$ then for each $m\in\mathbb{N}$.
\begin{equation}
\infty<\limsup_{x\rightarrow0^{+}}\frac{\log|f^{(-m)}(x)-P(x)|}{|\log x|}\,,
\label{momentseq4.10}
\end{equation}
for each polynomial $P$ of degree at most $m-1$;
\item[(III)] if $\displaystyle\liminf_{n\rightarrow\infty}\Re e\:\alpha_{n}=-\infty$ and $\displaystyle\limsup_{n\rightarrow\infty
}\Re e\:\alpha_{n}=\infty$, then for each $m\in\mathbb{N}$ both
\eqref{momentseq4.9} and \eqref{momentseq4.10} hold for each polynomial $P$ of degree at most $m-1$.
\end{enumerate}
\end{theorem}

\begin{proof}
For each $n$, find $f_{n}\in\mathcal{S}'_{+}$ such that $f_{n}(x)=x^{\alpha
_{n}}f(x)$ for $x\in(0,\infty)$. It is clearly that the sequence $\{f_{n}\}_{n=0}^{\infty}$ is
linearly independent on $(0,\infty)$ and consequently $\mathcal{N}_{0}$ has trivial intersection with the linear span of $\{f_{n}\}_{n=0}^{\infty}$.

Assume that every moment problem (\ref{momentseq4.7}) is solvable in
$\mathcal{S}(0,\infty)$. Fix $M>0$ and set $A_{M}$\thinspace$=\{n\in
\mathbb{N}:\:-M\leq\Re e$\thinspace$\alpha_{n}\leq M\}$. Since
$f\in\mathcal{S}_{+}^{\prime}$, there is $m\geq1$ and $\theta>m$ such that
$f^{(-m)}\in C[0,\infty)$, $f^{(-m)}(x)=O(x^{\theta})$ as $x\rightarrow\infty
$, and $f^{(-m)}(x)=O(x^{M+1/2})$ as $x\to0^{+}$. On the other hand (cf. \cite[Eq. (6.53), p. 302]{estrada-kanwal2002} or \cite[ Lemma 1.3]{campos}),
\begin{equation}
f_{n}^{(-m)}=\sum_{k=0}^{m}(-1)^{k}k!\binom{m}{k}\binom{\alpha_{n}}
{k}(x^{\alpha_{n}-k}f^{(-m)})^{(-k)}+P\ \ \ \mbox{ on }(0,\infty).
\label{momentseq4.11}%
\end{equation}
where, $P(x)$ is a polynomial of degree at most $m-1$. So, if $n\in A_{M}$,
one has that $f_{n}^{(-m)}(x)=O(x^{\theta+M})$ as $x\rightarrow\infty$ and
$f_{n}^{(-m)}(x)=Q(x)+O(x^{1/2})$ as $x\rightarrow0^{+}$ for some polynomial
$Q$ of degree at most $m-1$. In particular $f_{n}^{(-m)}\in C(0,\infty)$,
$f_{n}(x)=O(x^{\alpha})$ $(\mathrm{C},m)$, $x\rightarrow\infty$, and
$f_{n}=O(x^{-\alpha})$ $(\mathrm{C},m)$, $x\rightarrow0^{+}$, for $\alpha
\geq\max\{\theta+M-m,m-1/2\}$. The property (\ref{P2}) implies that this
cannot hold for infinitely many $n$, so that $A_{M}$ must be finite.

Suppose that $\displaystyle\lim_{n\rightarrow\infty}\Re e$\thinspace
$\alpha_{n}=\infty$ but that (\ref{momentseq4.9}) is false for some $m$, i.e.,
$f^{(-m)}(x)=Q(x)+o(x^{-\alpha_{n}})$ as $x\rightarrow\infty$, for every $n>0$, where $Q$ is a polynomial of degree less than $m$. Since 
$$
0=\sum_{k=0}^{m}(-1)^{k}k!\binom{m}{k}\binom{\alpha_{n}}
{k}(x^{\alpha_{n}-k}Q)^{(m-k)} \quad \mbox{ on } (0,\infty),$$
the formula
(\ref{momentseq4.10}) then yields $f_{n}^{(-m)}\in C(0,\infty)$
and $f_{n}(x)=O(1)$ $(\mathrm{C},m)$ as $x\rightarrow\infty$, for each
$n$. Furthermore, since $\Re e\:\alpha_{n}$ is bounded from below, we
also get $f_{n}(x)=O(x^{-\alpha})$ $(\mathrm{C},m)$, for some $\alpha$ (and a
possibly enlarged $m$), whence we deduce that (\ref{momentseq4.2}) (and hence (\ref{P2})) cannot hold. The
cases (II) and (III) can be treated in a similar fashion.

An analogous argument, with aid of the relation (cf. \cite[ Lemma 1.3]{campos})
\begin{equation*}
f^{(-m)}=\sum_{k=0}^{m}(-1)^{k}k!\binom{m}{k}
\left(\frac{g_{N}^{(-m)}}{F_{N}^{(k)}}\right)^{(-k)}+Q\ ,
\end{equation*}
where $F_{N}(x)=\sum_{n=0}^{N}b_{n}x^{\alpha_{n}}$, $g_{N}(x)=\sum_{n=0}^{N}b_{n}f_{n}$, and $Q$ a certain polynomial (depending on $F_{N}$) of degree at most $m-1$, shows that the conditions are sufficient for $(\ref{P2})$.
\end{proof}

Employing the same method as in proof of Theorem \ref{momentsth3}, one deduces:

\begin{proposition}
\label{momentsp2} Let $\{\alpha_{n}\}_{n=0}^{\infty}$ be a sequence of distinct complex numbers with the property there is $n_{0}$ such that $\Re e\:\alpha_{n}\neq\Re e\: \alpha_{m}$ for all distinct $n,m>n_{0}$. Suppose that \eqref{momentseq4.8} holds and there are
$m\geq1$ and $\sigma>1$ such that

\begin{enumerate}
\item[(I)] if $\displaystyle\lim_{n\rightarrow\infty}\Re e\:
\alpha_{n}=\infty$, then
\begin{equation}
-\infty<\limsup_{x\rightarrow\infty}\frac{\log\displaystyle\inf_{a\in
\lbrack1,\sigma]}|f^{(-m)}(ax)-P(ax)|}{\log x}\,, \label{momentseq4.12}
\end{equation}
for each polynomial $P$ of degree at most $m-1$;
\item[(II)] if $\displaystyle\lim_{n\rightarrow\infty}\Re e\:\alpha_{n}=-\infty,$ then
\begin{equation}
\liminf_{x\rightarrow0^{+}}\frac{\displaystyle\log\inf_{a\in\lbrack1,\sigma
]}|f^{(-m)}(ax)-P(ax)|}{|\log x|}<\infty\,, \label{momentseq4.13}
\end{equation}
for each polynomial $P$ of degree at most $m-1$;

\item[(III)] if $\displaystyle\liminf_{n\rightarrow\infty}\Re e\:\alpha_{n}=-\infty$ and $\displaystyle\limsup_{n\rightarrow\infty
}\Re e$\thinspace$\alpha_{n}=\infty$, then both \eqref{momentseq4.12} and
\eqref{momentseq4.13}) hold for each polynomial $P$ of degree at most $m-1$.
\end{enumerate}
Then every moment problem \eqref{momentseq4.7} has solutions $\phi
\in\mathcal{S}(0,\infty)$. If $f$ is a Darboux function, one may take $m=0$ (so that $P=0$) in
\eqref{momentseq4.12} and \eqref{momentseq4.13}.
\end{proposition}

\section{The measure weighted moment problem}\label{Section: measure}
We consider here the measure weighted moment problem
\begin{equation}\label{measeq1}
\int_{0}^{\infty} x^{\alpha_{n}}\phi(x)\mathrm{d}F(x)=a_{n}, \quad n\in\mathbb{N}\,,
\end{equation}
where we assume that $\{\alpha_{n}\}_{n=0}^{\infty}$ is sequence of \emph{real} numbers and $F$ is a function of local bounded variation on $(0,\infty)$ that satisfies
\begin{equation}\label{measeq2}
\int_{0}^{1}x^{\sigma_{1}} |\mathrm{d}F|(x)<\infty
\end{equation}
and
\begin{equation}\label{measeq3}
\int_{1}^{\infty}x^{-\sigma_{1}} |\mathrm{d}F|(x)<\infty\,,
\end{equation}
for some $\sigma_{1}>0$, where $|\mathrm{d}F|$ is the total variation measure of $\mathrm{d}F$. This ensures $\mathrm{d}F\in \mathcal{S}'(0,\infty)$.

We start with a necessary condition for the solvability of (\ref{measeq1}). 
\begin{proposition}\label{meas prop 1}
Suppose that every moment problem \eqref{measeq1} has solutions $\phi\in\mathcal{S}(0,\infty)$, then the numbers $\alpha_{n}$ are distinct, $|\alpha_{n}|\to\infty$ and there is $\sigma_{0}>0$ such that 

\begin{enumerate}
\item[(I)] if $\displaystyle\lim_{n\rightarrow\infty} \alpha_{n}=\infty$, then
\begin{equation}
\label{measeq4}
\int_{1}^{\infty}x^{-\sigma_{0}} |\mathrm{d}F|(x)=\infty\:;
\end{equation}
\item[(II)] if $\displaystyle\lim_{n\rightarrow\infty}\alpha_{n}=-\infty,$ then
\begin{equation}
\label{measeq5}
\int_{0}^{1}x^{\sigma_{0}} |\mathrm{d}F|(x)=\infty\:;
\end{equation}
\item [(III)] if $\displaystyle\liminf_{n\rightarrow\infty}\alpha_{n}=-\infty$ and $\displaystyle\limsup_{n\rightarrow\infty
}\alpha_{n}=\infty$, then both \eqref{measeq4} and
\eqref{measeq5} hold.
\end{enumerate}
\end{proposition}
\begin{proof} 
Let $f_{n}$ be an extension of $x^{\alpha_{n}}\mathrm{d}F(x)$ to $\mathcal{S}'_{+}$. Since \eqref{P1} holds, the $\alpha_{n}$'s should be distinct.
If $|\alpha_{n}|\leq M$ for all $n\in\mathbb{N}$, using the assumptions \eqref{measeq2} and \eqref{measeq3}, we would have
$$
\left|\sum_{n=0}^{N}b_{n}\int_{1}^{x}t^{\alpha_{n}}\mathrm{d}F(t)\right|= O(x^{M+\sigma_{1}}) \quad \mbox{as }x\to\infty$$
and 
\begin{equation}
\label{measeq6}
\left|\sum_{n=0}^{N}b_{n}\int_{1}^{x}t^{\alpha_{n}}\mathrm{d}F(t)\right|= O(x^{-M-\sigma_{1}}) \quad \mbox{as }x\to0^{+}\:,
\end{equation}
contradicting \eqref{P2} for $\{f_{n}\}_{n=0}^{\infty}$. If $\alpha_{n}\to\infty$, a bound \eqref{measeq6} must hold because $\alpha_{n}$ is bounded from below. If (\ref{measeq4}) were not valid, then we would have $\int_{1}^{\infty}t^{\alpha_{n}}|\mathrm{d}F|(t)<\infty$ for all $n\in\mathbb{N}$, leading to 
$$
\left|\sum_{n=0}^{N}b_{n}\int_{1}^{x}t^{\alpha_{n}}\mathrm{d}F(t)\right|= O(1)\:, \quad x\geq1\:.
 $$
This again contradicts \eqref{P2}. The cases (II) and (III) can be treated similarly.

\end{proof}

Note that either \eqref{measeq4} or \eqref{measeq5}  always yields that the support of $\mathrm{d}F$ is an infinite subset of $(0,\infty)$. If additionally the $\alpha_{n}$'s are distinct, the latter implies \eqref{P1} for any distribution sequence $\{f_n\}_{n=0}^{\infty}$ of $\mathcal{S}'_{+}$ that extends $x^{\alpha_{n}}\mathrm{d}F(x)$, that is,  $f_{n}(x)=x^{\alpha_{n}}\mathrm{d}F(x)$ on $(0,\infty)$.

We give the converse of Proposition \ref{meas prop 1} for non-negative measures. In fact, the next theorem gives a complete characterization of those non-negative measures $\mathrm{d}F$ for which \eqref{measeq1} is always solvable in $\mathcal{S}(0,\infty)$.

\begin{theorem}\label{meas thm 1}  Let $F$ be non-decreasing on $(0,\infty)$ satisfying \eqref{measeq2} and \eqref{measeq3}, and let $\{\alpha_{n}\}_{n=0}^{\infty}$ be a sequence of distinct real numbers. Then, every moment problem \eqref{measeq1} has solutions $\phi\in\mathcal{S}(0,\infty)$ if and only if $|\alpha_{n}|\to\infty$ and there is $\sigma_{0}>0$ such that 

\begin{enumerate}
\item[(I)] if $\displaystyle\lim_{n\rightarrow\infty}
\alpha_{n}=\infty$, then \eqref{measeq4} holds; 
\item[(II)] if $\displaystyle\lim_{n\rightarrow\infty}\alpha_{n}=-\infty,$ then \eqref{measeq5} holds;
\item [(III)] if $\displaystyle\liminf_{n\rightarrow\infty}\alpha_{n}=-\infty$ and $\displaystyle\limsup_{n\rightarrow\infty
}\alpha_{n}=\infty$, then both \eqref{measeq4} and
\eqref{measeq5} hold.
\end{enumerate}
\end{theorem}
\begin{proof} 
We only need to show the converse. We assume $\displaystyle\lim_{n\rightarrow\infty}\Re e\:
\alpha_{n}=\infty$, the proofs in remaining two cases are analogous. We may rearrange the sequence in increasing order $\alpha_{0}<\alpha_{1}<\dots<\alpha_{n}\to\infty$. We prove that \eqref{measeq4} implies \eqref{P2}, for an extension sequence of the $x^{\alpha_{n}}\mathrm{d}F(x)$ to $\mathcal{S}'_{+}$, by contraposition. Note that the $O$-bound \eqref{measeq6} holds as $x\to0^{+}$ with $-M=\alpha_0$;  so, if $\eqref{P2}$ fails, there are $m$, $\alpha>0$, and an infinite sequence of indices $N_{k}$ such that   
$$
\int_{1}^{x}\left(t^{\alpha_{N_{k}}}+\sum_{n=0}^{N_{k}-1}b_{n,k}t^{\alpha_{n}}\right)\left(1-\frac{t}{x}\right)^{m}\mathrm{d}F(t)=O(x^{\alpha}), \quad x\geq1\,.
$$
for some constants $b_{n,k}$. We have $\sum_{n=0}^{N_{k}-1}b_{n,k}t^{\alpha_{n}}=o(t^{\alpha_{N_{k}}})$. Therefore, 
$$
G(t)=\int_{1}^{x}t^{\alpha_{N_{k}}}\mathrm{d}F(t)\leq 2^{m}
\int_{1}^{2x}\left(1-\frac{t}{2x}\right)^{m}t^{\alpha_{N_{k}}}\mathrm{d}F(t)=O(x^{\alpha}), \quad x\geq1\,.
$$
Integrating by parts,
$$
\int_{1}^{x}t^{\alpha_{N_{k}}-\alpha-1}\mathrm{d}F(t)= \frac{G(x)}{x^{\alpha+1}}+(\alpha+1)\int_{1}^{x}\frac{G(t)}{t^{\alpha+2}}\mathrm{d}t=O(1)\:,
$$
which implies that (\ref{measeq4}) cannot hold because $\alpha_{N_{k}}\to\infty$.
\end{proof}

If the measure $\mathrm{d}F$ vanishes on $(0,\lambda_{0})$, then \eqref{measeq5} cannot hold, which excludes the cases (II) and (III) from Proposition \ref{meas prop 1} and Theorem \ref{meas thm 1}. Also, the abscissa of absolute convergence of the Mellin transform of $\int_{0}^{\infty}x^{-s}\mathrm{d}F(x)$, denoted as $\sigma_{a}$, is the infimum of those $\sigma_{1}$ for which (\ref{measeq3}) holds. Under our assumption (\ref{measeq3}), $\sigma_{a}<\infty$, but it may be equal to $-\infty$.

\begin{corollary}
\label{meas cor 1} 
Let $F$ be a non-decreasing function of at most polynomial growth that vanish on  $(-\infty,\lambda)$ for some $\lambda>0$, and let $\{\alpha_{n}\}_{n=0}^{\infty}$ be a sequence of distinct real numbers. Then, every moment problem \eqref{measeq1} has solutions $\phi\in\mathcal{S}(0,\infty)$ if and only if $\alpha_{n}\to\infty$ and the abscissa of convergence of the Mellin transform of $\mathrm{d}F$ is finite, that is, $\sigma_{a}>-\infty$.

\end{corollary}

\begin{example} Let $\{\lambda_{k}\}_{k=1}^{\infty}$ be a non-decreasing sequence of positive real numbers and let $\{c_k\}_{k=0}^{\infty}$ be a non-negative sequence. According to Corollary \ref{meas cor 1}, every arbitrary moment problem 
$$
a_n=\sum_{k=1}^{\infty} c_{k}\lambda_{k}^{\alpha_{n}}\phi(k), \quad n\in\mathbb{N}.
$$ 
has solutions $\phi\in\mathcal{S}(0,\infty)$ if and only if $\alpha_{n}\to\infty$ and the Dirichlet series $F(s)=\sum_{k=1}^{\infty} c_{k}\lambda^{-s}$ has finite abscissa of convergence. In particular, if $\alpha_{n}\to\infty$, moment problems such as
$$
a_n=\sum_{k=1}^{\infty} \phi(k)k^{\alpha_{n}}, \quad n\in\mathbb{N}\:, \quad \mbox{and }\quad  a_n=\sum_{p \text{ prime}} \phi(p)p^{\alpha_{n}}, \quad n\in\mathbb{N}\:,
$$ 
are always solvable in $\mathcal{S}(0,\infty)$. 
\end{example}

\section{Density of the set of solutions of moment
problems\label{Section: Density of the set of solutions of moment problems}}

In order to study vector valued moment problems, it is convenient to consider
first several results on the density of some linear manifolds in a general
topological vector space.

Let $E$ be a locally convex topological vector space. It is well known that a
linear functional $f:E\longrightarrow\mathbb{C}$ is continuous if and only if
the linear subspace $\ker f=\{x\in E:f\left(  x\right)  =0\}$ is closed, or,
equivalently, $f$ is discontinuous if and only if $\ker f$ is dense in $E.$ We
would like to consider the corresponding situation when not one but several
linear functionals are given. In the following we shall employ the notation
$\ker\left(  f_{1},\ldots,f_{n}\right)  =\bigcap_{k=1}^{n}\ker f_{k}
.$

\begin{definition}
\label{Definition: Completely discontinuous}Let $f_{1},\ldots,f_{n}$ be linear
functionals on the locally convex topological vector space $E.$ We say that
they are completely discontinuous if the only linear combination $\sum
_{k=1}^{n}c_{k}f_{k}$ that is continuous is the one with $c_{k}=0$ for all
$k.$\smallskip
\end{definition}

If we denote by $\pi$ the projection from the \emph{algebraic} dual space
$E_{\text{alg}}^{\prime}$ onto $E_{\text{alg}}^{\prime}/E^{\prime},$ then
$f_{1},\ldots,f_{n}$ are completely discontinuous if and only if $\pi(f_{1}),$
$\ldots,$ $\pi(f_{n})$ are linearly independent.

\begin{proposition}
\label{Prop. V1}Let $f_{1},\ldots,f_{n}$ be $n$ linearly independent linear
functionals on the locally convex topological vector space $E.$ Let $k$ be the
dimension of the vector subspace $G$ of $\mathbb{C}^{n}$ formed by those
vectors $\left(  c_{1},\ldots,c_{n}\right)  $ such that $\sum_{i=1}^{n}
c_{i}f_{i}$ is continuous. Then
\begin{equation*}
\operatorname*{codim}\nolimits_{E}\overline{\ker\left(  f_{1},\ldots
,f_{n}\right)  }=k\,. 
\end{equation*}
In particular, $\ker\left(  f_{1},\ldots,f_{n}\right)  $ is dense in $E$ if
and only if $f_{1},\ldots,f_{n}$ are completely discontinuous.
\end{proposition}

\begin{proof}
Indeed, denote by $m$ the codimension of $\overline{\ker\left(  f_{1}
,\ldots,f_{n}\right)  }$ in $E.$ We shall first show that $m\geq k.$ This is
obvious, of course, if $k=0,$ so let us suppose that $k>0.$ Then if
$\mathbf{c}_{j}=\left(  c_{j,i}\right)  _{i=1}^{n},$ $1\leq j\leq k,$\ are a
basis of $G,$ then the $k$ functionals $g_{j}=\sum_{i=1}^{n}c_{j,i}f_{i},$
$1\leq j\leq k$ are continuous, and they are also linearly independent,
because the $f_{j}$'s are. Since $\ker\left(  f_{1},\ldots,f_{n}\right)
\subset\ker\left(  g_{1},\ldots,g_{k}\right)  ,$ and the latter space is
closed, we obtain that $\overline{\ker\left(  f_{1},\ldots,f_{n}\right)
}\subset\ker\left(  g_{1},\ldots,g_{k}\right)  ,$ and thus $m\geq k.$

To show that $m\leq k$ we may assume that $m>0.$ Since $\overline{\ker\left(
f_{1},\ldots,f_{n}\right)  }$ is a closed subspace of codimension $m,$ we can
find $m$ linearly independent continuous functionals $g_{1},\ldots,g_{m}$ such
that $\overline{\ker\left(  f_{1},\ldots,f_{n}\right)  }=\ker\left(
g_{1},\ldots,g_{m}\right)  .$ The fact that $\ker\left(  f_{1},\ldots
,f_{n}\right)  \subset\ker g_{j}$ for any $j$ yields that $g_{j}$ is a linear
combination of $f_{1},\ldots,f_{n},$ say $g_{j}=\sum_{i=1}^{n}c_{j,i}f_{i};$
then the vectors $\mathbf{c}_{j}=\left(  c_{j,i}\right)  _{i=1}^{n}$ for
$1\leq j\leq m$ are linearly independent in $G,$ and thus $m\leq k.$
\end{proof}

Observe, furthermore, that if $f_{1},\ldots,f_{n}$ are linearly independent,
then the dimension of the vector space generated by $\pi(f_{1}),\ldots
,\pi(f_{n})$ is precisely $n-k.$

If $f_{1},\ldots,f_{n}$ are linearly independent, then the map from $E$ to
$\mathbb{C}^{n}$ given by $x\mapsto\left(  \left\langle f_{j},x\right\rangle
\right)  _{j=1}^{n}$ is surjective. Therefore we obtain the following result
on \emph{finite} moment problems.

\begin{corollary}
\label{Cor. V2}If $f_{1},\ldots,f_{n}$ are completely discontinuous, then for
any vector $\left(  a_{j}\right)  _{j=1}^{n}$ of $\mathbb{C}^{n},$ the set of
solutions of the moment problem
\begin{equation*}
\left\langle f_{j},x\right\rangle =a_{j}\,,\ \ \ \ 1\leq j\leq n\,,
\end{equation*}
is dense in $E.$
\end{corollary}

Sometimes it is necessary to employ the following simple result.

\begin{lemma}
\label{Lemma V.1}Let $E$ be a locally convex topological vector space and let
$F$ be a closed subspace of finite codimension. Then the linear functionals
$f_{1},\ldots,f_{n}$ are completely discontinuous in $E$ if and only if their
restrictions to $F$ are.
\end{lemma}

In general the intersection of an infinite sequence of dense linear manifolds
does not have to be dense, so that the result of the Corollary \ref{Cor. V2}
does not hold for infinite moment problems. We shall show that when
$E=\mathcal{S}\left(  0,\infty\right)  $ and the linear functionals are
Ces\`{a}ro admissible then the density of the set of solutions holds in many
cases, but before we do this we give an example when the set of solutions is
not dense.

\begin{example}
\label{Example V.1}Let $E$ be the normed space formed by the polynomials in
one variable, with norm $\left\Vert p\right\Vert =\max_{\left\vert
t\right\vert \leq1}p\left(  t\right)  .$ Consider the sequence of functionals
$\{f_{k}\}_{k=1}^{\infty},$ $f_{k}\left(  t\right)  =\delta^{\left(  k\right)
}\left(  t\right)  .$ For each $n,$ $f_{1},\ldots,f_{n}$ are completely
discontinuous, and as the corollary predicts, the set $S_{n}=\left\{  p\in
E:\left\langle f_{j},p\right\rangle =a_{j}\,,\ 1\leq j\leq n\right\}  $ is
dense in $E$ for any constants $a_{1},\ldots,a_{n}.$ However, for many
infinite sequences $\{a_{k}\}_{k=1}^{\infty}$ we have that the set $\left\{  p\in
E:\left\langle f_{j},p\right\rangle =a_{j}\,,\ j\geq1\right\}  $ is empty, and
when not empty, it is an affine subspace of dimension $1.$ Thus $\bigcap
_{n=1}^{\infty}S_{n}$ is never dense in $E.$\smallskip
\end{example}

Suppose now that $E$ is a Fr\'{e}chet space whose topology is given by the
basis of increasing seminorms $\left\{  p_{k}\right\}  _{k=1}^{\infty}.$
Consider a family $\{f_{k}\}_{k=1}^{\infty}$ of continuous linear functionals.
We shall say that the sequence of blocks $\left\{  \{f_{j}\}_{j=N_{k}}
^{N_{k}+1}\right\}  _{k=0}^{\infty}$ is \emph{strictly admissible} with
respect to the sequence of seminorms $\left\{  p_{k}\right\}  _{k=1}^{\infty}$
if $\{N_{k}\}_{k=0}^{\infty}$\ is an increasing sequence of integers with
$N_{0}=1$ such that the following two conditions are satisfied:

\begin{enumerate}
\item $\left\{  f_{1},\ldots,f_{N_{k}}\right\}  $ are continuous with respect
to $p_{k};$

\item $\left\{  f_{N_{k}+1},f_{N_{k}+2},\ldots\right\}  $ are completely
discontinuous with respect to $p_{k}.$ \smallskip
\end{enumerate}

We can then prove the density of the set of solutions of certain infinite
moment problems.

\begin{proposition}
\label{Prpp. V2}Let $\left\{  \{f_{j}\}_{j=N_{k}}^{N_{k}+1}\right\}
_{k=0}^{\infty}$ be strictly admissible with respect to $\left\{
p_{k}\right\}  _{k=1}^{\infty}.$ Let $\{a_{k}\}_{k=N_{1}+1}^{\infty}$ be an
arbitrary sequence of complex numbers. Then the set of solutions of the moment
problem
\begin{equation}
\left\langle f_{j},x\right\rangle =a_{j}\,,\ \ \ \ j\geq N_{1}+1\,,
\label{V.3}
\end{equation}
is dense in the seminormed space $\left(  E,p_{1}\right)  .$ Actually for any
$\varepsilon>0$ and any $x_{0}\in E$ there is a solution of (\ref{V.3}) such
that
\begin{equation*}
p_{1}\left(  x-x_{0}\right)  \leq\varepsilon\,,\ \ \ \ \left\langle
f_{j},x\right\rangle =\left\langle f_{j},x_{0}\right\rangle \,,\ 1\leq j\leq
N_{1}\,.
\end{equation*}

\end{proposition}

\begin{proof}
Let $\{\varepsilon_{k}\}_{k=1}^{\infty}$ be a sequence with $\varepsilon
_{k}>0$ for all $k$ and with $\sum_{k=1}^{\infty}\varepsilon_{k}=\varepsilon.$
Considering the set $f_{N_{1}+1},\ldots,f_{N_{2}},$ which is completely
discontinuous with respect to $p_{1},$ the Corollary \ref{Cor. V2} and the
Lemma \ref{Lemma V.1}\ show the existence of $x_{1}$ such that $p_{1}\left(
x_{1}\right)  <\varepsilon_{1},$ $\left\langle f_{j},x_{1}\right\rangle =0$
for $1\leq j\leq N_{1},$ while $\left\langle f_{j},x_{1}\right\rangle
=a_{j}-\left\langle f_{j},x_{0}\right\rangle $ for $N_{1}+1\leq j\leq N_{2}.$
Proceeding in a recursive fashion, for each $k\geq2$ we can find $x_{k}$ such
that $p_{k}\left(  x_{k}\right)  <\varepsilon_{k},$ $\left\langle f_{j}
,x_{k}\right\rangle =0$ for $1\leq j\leq N_{k},$ while $\left\langle
f_{j},x_{k}\right\rangle =a_{j}-\sum_{i=0}^{k-1}\left\langle f_{j}
,x_{i}\right\rangle $ for $N_{k}+1\leq j\leq N_{k+1}.$

The series $\sum_{k=0}^{\infty}x_{k}$ converges in $E$ since for any $q$ we
have%
\[
\sum_{k=0}^{\infty}p_{q}\left(  x_{k}\right)  \leq\sum_{k=0}^{q-1}p_{q}\left(
x_{k}\right)  +\sum_{k=q}^{\infty}p_{q}\left(  x_{k}\right)  \leq\sum
_{k=0}^{q-1}p_{q}\left(  x_{k}\right)  +\sum_{k=q}^{\infty}\varepsilon
_{k}<\infty\,.
\]
Let $x=\sum_{k=0}^{\infty}x_{k}.$ Then
\[
p_{1}\left(  x-x_{0}\right)  \leq\sum_{k=1}^{\infty}p_{1}\left(  x_{k}\right)
\leq\sum_{k=1}^{\infty}\varepsilon_{k}=\varepsilon\,,
\]
while by continuity $\left\langle f_{j},x\right\rangle =\sum_{i=0}^{\infty
}\left\langle f_{j},x_{i}\right\rangle ,$ which is actually a finite sum; this
gives gives $\left\langle f_{j},x\right\rangle =\left\langle f_{j}
,x_{0}\right\rangle \,,$\ for $1\leq j\leq N_{1}$ and if $k\geq1,$
\[
\left\langle f_{j},x\right\rangle =\sum_{i=0}^{k}\left\langle f_{j}
,x_{i}\right\rangle =\left\langle f_{j},x_{k}\right\rangle +\sum_{i=0}
^{k-1}\left\langle f_{j},x_{i}\right\rangle =a_{j}\,,
\]
for $N_{k}+1\leq j\leq N_{k+1},$ as required.
\end{proof}

Let us now go back to moment problems in the space $E=\mathcal{S}\left(
0,\infty\right)  .$ If a sequence of functionals $\{f_{k}\}_{k=0}^{\infty}$ is
the restriction to $\left(  0,\infty\right)  $ of a sequence that is
Ces\`{a}ro admissible, and $\left\{  p_{k}\right\}  _{k=1}^{\infty}$ is an
increasing sequence of continuous seminorms of $E$ that gives the topology, in
general it is not possible to arrange $\{f_{k}\}_{k=0}^{\infty}$ in blocks to
obtain strict admissibility. However, we can construct sequences $\left\{
\widetilde{p}_{k}\right\}  _{k=1}^{\infty},$ $\{g_{k}\}_{k=0}^{\infty},$ and
$\{N_{k}\}_{k=0}^{\infty}$\ such that

\begin{enumerate}
\item [(\textbf{1})]$\left\{  \widetilde{p}_{k}\right\}  _{k=1}^{\infty}$ is also an
increasing sequence of continuous seminorms that gives the topology of $E$;

\item [(\textbf{2})]$\left\{  \{g_{j}\}_{j=N_{k}}^{N_{k}+1}\right\}  _{k=0}^{\infty}$ is
strictly admissible with respect to $\left\{  \widetilde{p}_{k}\right\}
_{k=1}^{\infty};$

\item [(\textbf{3})] there is a linear bijective map $T$ from $\mathbb{C}^{\mathbb{N}}$ to
itself\footnote{The map $T$ is actually bicontinuous with respect to the
topology of pointwise convergence in $\mathbb{C}^{\mathbb{N}}.$} such that
$T\left(  \left\{  f_{j}\right\}  _{j=0}^{\infty}\right)  $ $=\left\{
g_{j}\right\}  _{j=0}^{\infty}.$ Actually there are increasing sequences
$m_{0}=0<m_{1}<m_{2}<\cdots,$ such that the vector $\left(  g_{j}\right)
_{j=m_{k}+1}^{m_{k+1}}$ is obtained by multiplying $\left(  f_{j}\right)
_{j=m_{k}+1}^{m_{k+1}}$ by an invertible matrix.\smallskip
\end{enumerate}

The construction is as follows. Naturally we put $N_{0}=1.$ Next, take
$\widetilde{p}_{1}=p_{1}.$ Since $\{f_{k}\}_{k=0}^{\infty}$ is Ces\`{a}ro
admissible, we can find integers $n_{1}$ and $l_{1},$ with $n_{1}\leq l_{1}$
such that $f_{0},\ldots,f_{n_{1}}$ are continuous with respect to
$\widetilde{p}_{1}$ while $f_{l_{1}+1},$ $f_{l_{1}+2},$ $f_{l_{1}+3},$
$\ldots$ are completely discontinuous with respect to $\widetilde{p}_{1};$ we
can take them so that $n_{1}$ is the maximum while $l_{1}$ is the minimum with
this property.\ If $l_{1}=n_{1}$ we take $N_{1}=n_{1}.$ If $l_{1}=n_{1}+q,$
$q\geq1,$ we can find linear combinations
\begin{equation*}
g_{n_{1}+j}=\sum_{i=1}^{q}\alpha_{j,i}f_{n_{1}+i}+\sum_{k=n_{1}+q+1}^{m_{1}
}\beta_{j,i}f_{k}\,, 
\end{equation*}
for $1\leq j\leq q$ such that the matrix $\left(  \alpha_{j,i}\right)
_{j,i=1}^{q}$ is invertible, and for some $N_{1},$ with $n_{1}<N_{1}\leq
l_{1}$ the functionals $f_{0},$ $\ldots,$ $f_{n_{1}},$ $g_{n_{1}+1},$ $\ldots$
$g_{N_{1}}$ are continuous with respect to $\widetilde{p}_{1}$ while
$g_{N_{1}+1},$ $\ldots,$ $g_{l_{1}},$ $f_{l_{1}+1},$ $f_{l_{1}+2},$
$f_{l_{1}+3},$ $\ldots$ are completely discontinuous with respect to
$\widetilde{p}_{1}.$ Put $g_{j}=f_{j}$ for $0\leq j\leq n_{1}$ and for
$l_{1}+1\leq j\leq m_{1}.$ The vector $\left(  g_{j}\right)  _{j=1}^{m_{1}}$
is obtained by multiplying $\left(  f_{j}\right)  _{j=1}^{m_{1}}$ by an
invertible matrix.

We then find a seminorm $\widetilde{p}_{2}$ among the $p_{k}$'s for $k\geq
r_{1}$ such that $g_{0},\ldots,g_{m_{1}}$ and $f_{m_{1}+1},\ldots,f_{n_{2}}$
are continuous with respect to $\widetilde{p}_{2}$ while $f_{l_{2}+1},$
$f_{l_{2}+2},$ $f_{l_{2}+3},$ $\ldots$ are completely discontinuous with
respect to $\widetilde{p}_{2}$ for some integers $n_{2}$ and $l_{2},$ with
$m_{1}+1\leq n_{2}\leq l_{2}.$ Then we employ the same procedure to find
$N_{2}$ and $m_{2}$ such that $n_{2}\leq N_{2}\leq l_{2}\leq m_{2}$ and linear
combinations $g_{j},$ $m_{1}+1\leq j\leq m_{2}$ of the $f_{k}$'s, $m_{1}+1\leq
k\leq m_{2},$ so that $\left(  g_{j}\right)  _{j=m_{1}+1}^{m_{2}}$ is obtained
by multiplying $\left(  f_{j}\right)  _{j=m_{1}+1}^{m_{2}}$ by an invertible
matrix, $g_{0},$ $\ldots$ $g_{N_{2}}$ are continuous with respect to
$\widetilde{p}_{2}$ while $g_{N_{2}+1},$ $g_{N_{2}+2},$ $g_{N_{2}+3},$
$\ldots$ are completely discontinuous with respect to $\widetilde{p}_{2}.$

We may then proceed inductively, constructing seminorms $\widetilde{p}_{k},$
integers $n_{k},$ $N_{k},$ $l_{k},$ and $m_{k}$ with $m_{k-1}+1\leq n_{k}\leq
N_{k}\leq l_{k}\leq m_{k},$ define $\left(  g_{j}\right)  _{j=m_{k-1}
+1}^{m_{k}}$ by multiplying $\left(  f_{j}\right)  _{j=m_{k-1}+1}^{m_{k}}$ by
an appropriate invertible matrix, so that $g_{1},$ $\ldots$ $g_{N_{k}}$ are
continuous with respect to $\widetilde{p}_{k}$ while $g_{N_{k}+1},$
$g_{N_{k}+2},$ $g_{N_{k}+3},$ $\ldots$ are completely discontinuous with
respect to $\widetilde{p}_{k}.$ The three conditions above are then clearly satisfied.

Notice that condition (\textbf{3}) in our construction implies that the moment
problems
\begin{equation*}
\left\langle f_{j},\phi\right\rangle =a_{j}\,,\ \ \ \ j\in\mathbb{N}\,,
\end{equation*}
have solution for \emph{all} sequences $\{a_{k}\}_{k=0}^{\infty}$ if and only
if the same is true of all the moment problems
\begin{equation*}
\left\langle g_{j},\phi\right\rangle =b_{j}\,,\ \ \ \ j\in\mathbb{N}\,,
\end{equation*}
for arbitrary sequences $\{b_{k}\}_{k=0}^{\infty}.$ We can also use the
density of the set of solutions of one moment problem to obtain the
corresponding density of the set of solutions of the other, as we explain in
precise terms in the next proposition.

\begin{proposition}
\label{Prop. V3}Let $\{f_{k}\}_{k=0}^{\infty}$ be a sequence of distributions
of the space $\mathcal{S}_{+}$ that satisfies conditions \eqref{P1} and \eqref{P2}. Let
$p$ be any continuous seminorm in $\mathcal{S}\left(  0,\infty\right)  .$ Then
there exists $m\in\mathbb{N}$ such that if $\{a_{k}\}_{k=m}^{\infty}$ is an
arbitrary sequence of complex numbers, then the set of solutions of the moment
problem
\begin{equation}
\left\langle f_{j},\phi\right\rangle =a_{j}\,,\ \ \ \ j\geq m\,, \label{V.8}
\end{equation}
is dense in the seminormed space $\left(  \mathcal{S}\left(  0,\infty\right)
,p\right)  .$ If the sequence $\{f_{k}\}_{k=0}^{\infty}$ is completely
discontinuous with respect to $p$ then we can take $m=0.$
\end{proposition}

\begin{proof}
Let $\left\{  p_{k}\right\}  _{k=1}^{\infty}$ be an increasing sequence of
continuous seminorms of $\mathcal{S}\left(  0,\infty\right)  $ that gives the
topology, with $p_{1}=p.$ We can then construct the sequence $\left\{
g_{j}\right\}  $ so that conditions (\textbf{1}), (\textbf{2}), and (\textbf{3}) are satisfied. Observe that
for any $k$ the set of solutions of the moment problem $\left\langle
f_{j},\phi\right\rangle =0,$ for $j\geq m_{k},$ is exactly the same as the set
of solutions of $\left\langle g_{j},\phi\right\rangle =0,$ for $j\geq m_{k},$
and since $\left\{  \{g_{j}\}_{j=N_{k}}^{N_{k}+1}\right\}  _{k=0}^{\infty}$ is
strictly admissible with respect to $\left\{  \widetilde{p}_{k}\right\}
_{k=1}^{\infty},$ Proposition \ref{Prpp. V2} gives that such set of solutions is dense in
$\mathcal{S}\left(  0,\infty\right)  $ with the topology given by the seminorm
$\widetilde{p}_{k}.$ Hence, if we take $m=m_{1}$ we obtain that the set of
solutions of $\left\langle f_{j},\phi\right\rangle =0$ for $j\geq m\,$ is
dense in $\left(  \mathcal{S}\left(  0,\infty\right)  ,p\right)  .$ Therefore,
if $\{a_{k}\}_{k=m}^{\infty}$ is an arbitrary sequence of complex numbers,\ by
translating by a particular solution of (\ref{V.8}) -- particular solution
that exists because \eqref{P1} and \eqref{P2} are satisfied -- we obtain that the set of
solutions of (\ref{V.8}) is likewise dense in $\left(  \mathcal{S}\left(
0,\infty\right)  ,p\right)  .$ That we can take $m=0$ if $\{f_{k}
\}_{k=0}^{\infty}$ is completely discontinuous with respect to $p$ should be
clear from our construction.
\end{proof}

The next related result, which will be needed in our analysis of the vector
moment problems, follows by the same arguments.

\begin{proposition}
\label{Prop. V4}Let $\{f_{k}\}_{k=0}^{\infty}$ be a sequence of distributions
of the space $\mathcal{S}_{+}$ that satisfies conditions \eqref{P1} and \eqref{P2}. Let
$p$ be any continuous seminorm in $\mathcal{S}\left(0,\infty\right).$ Then
there exists $m\in\mathbb{N}$ such that if $\{a_{k}\}_{k=m}^{\infty}$ is an
arbitrary sequence of complex numbers, then for each $\varepsilon>0$ the
moment problem
\begin{equation*}
\left\langle f_{j},\phi\right\rangle =a_{j}\,,\ \ \ \ j\geq
m\,,\ \ \ \ \left\langle f_{j},\phi\right\rangle =0\,,\ \ \ \ 0\leq j<m\,,
\end{equation*}
has a solution with $p\left(  \phi\right)  <\varepsilon.$
\end{proposition}

\section{Vector moment problems\label{Section: Vector moment problems}}

We now consider some vector moment problems.

Let $\mathcal{X}$ be a Fr\'{e}chet space. Let $\Vert\ \Vert_{1}\leq
\Vert\ \Vert_{2}\leq\Vert\ \Vert_{3}\ldots$ be a sequence of seminorms of
$\mathcal{X}$ that generate its topology. Let $\{\mathbf{a}_{k}\}_{k=0}
^{\infty}$ be an arbitrary sequence of elements of $\mathcal{X}$ and let
$\{f_{j}\}_{j=0}^{\infty}$ be a sequence of distributions of the space
$\mathcal{S}_{+}.$ We wish to study the problem of finding a rapidly
decreasing smooth function $\mathbf{\phi}:\mathbb{R}\rightarrow\mathcal{X}$
with support in $[0,\infty)$ such that
\begin{equation}
\left\langle f_{j},\mathbf{\phi}\right\rangle =\mathbf{a}_{j}
\,,\;\;j=0,1,2,3,\ldots\,. \label{VM.2.1}
\end{equation}
Notice that asking $\mathbf{\phi}$ to be a rapidly decreasing smooth function
means that $\mathbf{\phi}\in\mathcal{S}(\mathbb{R},\mathcal{X})\cong
\mathcal{S}(\mathbb{R})\,\widehat{\otimes}\,\mathcal{X}.$ In general
\cite{treves}, $\mathbf{\psi}$ belongs to $\mathcal{S}(\mathbb{R}
^{n},\mathcal{X})$ if and only if for each $\mathbf{k,m}\in\mathbb{N}^{n}$ the
set $\{\mathbf{x}^{\mathbf{k}}\mathbf{D}^{\mathbf{m}}\mathbf{\psi}
(\mathbf{x}):\mathbf{x}\in\mathbb{R}^{n}\}$ is bounded in $\mathcal{X}.$ For a
Fr\'{e}chet space this means that
\begin{equation}
\Vert\mathbf{D}^{\mathbf{m}}\mathbf{\psi(x)}\Vert_{q}=O(|\mathbf{x}
|^{-k})\,,\;\mathrm{as\;}|\mathbf{x}|\rightarrow\infty\,, \label{VM.2.2}
\end{equation}
for each $\mathbf{m}\in\mathbb{N}^{n}$ and $q,k\in\mathbb{N}.$ Notice also
that $\mathcal{S}((0,\infty),\mathcal{X})\cong\mathcal{S}(0,\infty
)\,\widehat{\otimes}\,\mathcal{X}$ are the elements of $\mathcal{S}
(\mathbb{R},\mathcal{X})$ with support in $[0,\infty).$

If $\{f_{j}\}_{j=0}^{\infty}$ is a sequence of distributions of the space
$\mathcal{S}_{+}$ that satisfies conditions \eqref{P1} and \eqref{P2}, then we can find
functions $\rho_{k}\in\mathcal{S}(0,\infty)$ such that
\begin{equation}
\left\langle f_{j},\rho_{k}\right\rangle =\delta_{j,k}\,,\;\;j=0,1,2,3,\ldots
\ , \label{VM.2.4}
\end{equation}
where $\delta_{j,k}=0,$ $j\neq k,$ $\delta_{k,k}=1,$ is the Kronecker delta.
One is tempted to try to solve (\ref{VM.2.1}) by setting $\mathbf{\phi
}(x)=\sum_{k=0}^{\infty}\rho_{k}(x)\mathbf{a}_{k}.$ However, this series could
be divergent and even if convergent the sum might not belong to $\mathcal{S}
((0,\infty),\mathcal{X}).$ However, as we shall show, if we choose the
$\rho_{k}$'s carefully then the series converges and gives an element of
$\mathcal{S}((0,\infty),\mathcal{X}).$

\begin{theorem}
\label{Theorem Vec 1}Let $\mathcal{X}$ be a Fr\'{e}chet space and let
$\{\mathbf{a}_{k}\}$ be an arbitrary sequence of elements of $\mathcal{X}.$
Let $\{f_{k}\}_{k=0}^{\infty}$ be a sequence of distributions of the space
$\mathcal{S}_{+}$ that satisfies conditions \eqref{P1} and \eqref{P2}. Then the moment
problem
\begin{equation}
\left\langle f_{j},\mathbf{\phi}\right\rangle =\mathbf{a}_{j}\,,\;\;j\in
\mathbb{N}\ . \label{VM.5}
\end{equation}
has solutions $\mathbf{\phi}\in\mathcal{S}((0,\infty),\mathcal{X}).$
\end{theorem}

\begin{proof}
Let
\begin{equation}
Q_{n}=\left\Vert \mathbf{a}_{n}\right\Vert _{n}\,, \label{VM.6}
\end{equation}
so that $\left\Vert \mathbf{a}_{j}\right\Vert _{n}\leq Q_{j}$ if $j\geq n.$
Let us choose $\varepsilon_{n}>0$ so that $\sum_{n=1}^{\infty}\varepsilon
_{n}Q_{n}<\infty.$ Let $\left\{  p_{k}\right\}  _{k=1}^{\infty}$ be an
increasing sequence of continuous seminorms of $\mathcal{S}\left(
0,\infty\right)  $ that gives the topology.

Employing Proposition \ref{Prop. V4} we can find a sequence of positive
integers $\left\{  m_{i}\right\}  _{i=1}^{\infty},$ which we may suppose
increasing, and for $k\geq m_{i}$\ solutions $\rho_{k}^{\{i\}}$ of the moment
problem (\ref{VM.2.4}) with $p_{i}\left(  \rho_{k}^{\{i\}}\right)
\leq\varepsilon_{k}.$ If we now put $\rho_{k}=\rho_{k}^{\{i\}}$ for $m_{i}\leq
k\leq m_{i+1}-1,$ and take any solutions $\rho_{k}$ for $k<m_{1},$ then the
series $\sum_{k=0}^{\infty}\rho_{k}(x)\mathbf{a}_{k}$ converges in
$\mathcal{S}(0,\infty)\,\widehat{\otimes}\,\mathcal{X}\cong\mathcal{S}
((0,\infty),\mathcal{X}).$ Indeed, convergence of the series in the tensor
product would follow \cite{treves}\ if we show that for each $N$ and $M$ the
series $\sum_{k=0}^{\infty}p_{N}\left(  \rho_{k}\right)  \left\Vert
\mathbf{a}_{k}\right\Vert _{M}$ converges. But by taking $K=\max\left\{
M,m_{N}\right\}  $ we obtain
\begin{align*}
\sum_{k=0}^{\infty}p_{N}\left(  \rho_{k}\right)  \left\Vert \mathbf{a}
_{k}\right\Vert _{M}  &  \leq\sum_{k=0}^{K-1}p_{N}\left(  \rho_{k}\right)
\left\Vert \mathbf{a}_{k}\right\Vert _{M}+\sum_{k=K}^{\infty}p_{N}\left(
\rho_{k}\right)  \left\Vert \mathbf{a}_{k}\right\Vert _{M}\\
&  \leq\sum_{k=0}^{K-1}p_{N}\left(  \rho_{k}\right)  \left\Vert \mathbf{a}
_{k}\right\Vert _{M}+\sum_{k=K}^{\infty}\varepsilon_{k}Q_{k}<\infty\,.
\end{align*}
If we now put $\mathbf{\phi}(x)=\sum_{k=0}^{\infty}\rho_{k}(x)\mathbf{a}_{k}$
we obtain
\[
\left\langle f_{j},\mathbf{\phi}\right\rangle =\sum_{k=0}^{\infty}\left\langle
f_{j},\rho_{k}(x)\right\rangle \mathbf{a}_{k}=\sum_{k=0}^{\infty}\delta
_{j,k}\mathbf{a}_{k}=\mathbf{a}_{j}\,,
\]
because of the convergence and the continuity of the $f_{j}$'s.
\end{proof}
Needless to say, the conditions \eqref{P1} and \eqref{P2} are also necessary for the solvability of the vector moment problem, as follows from Theorem \ref{momentsth2}.
\section{Moment problems in several
variables\label{Section: Moment problems in several variables}}

In this section we consider moment problems for functions of several
variables. Indeed, let $V$ be an open cone with vertex at the origin in
$\mathbb{R}^{d},$ that is, $V$ is an open set such that $\lambda\mathbf{x}\in
V$ whenever $\lambda>0$ and $\mathbf{x}\in V.$ Let $\mathcal{X}$ be a Fr\'{e}chet space. We denote by $\mathcal{S}(V,\mathcal{X})$
the space of elements of $\mathcal{S}(\mathbb{R}^{d},\mathcal{X})$ with support contained
in $\overline{V}.$ If $\{F_{n}\}_{n=0}^{\infty}$ is a sequence of
distributions of the space $\mathcal{S}^{\prime}(\mathbb{R}^{d})$ and
$\{\mathbf{a}_{n}\}_{n=0}^{\infty}$ is a sequence of vectors of $\mathcal{X}$, we would like to
study the existence of solutions of the vector moment problem
\begin{equation}
\left\langle F_{n}\left(  \mathbf{x}\right)  ,\phi\left(  \mathbf{x}\right)
\right\rangle =\mathbf{a}_{n}\,,\ \ \ n\in\mathbb{N}\,, \label{SV.1}
\end{equation}
in the space $\mathcal{S}(V,\mathcal{X}).$

We shall denote by $\mathbb{S}$ the unit sphere of $\mathbb{R}^{d},$ and if
$U$ is an open subset of $\mathbb{S}$ then $\mathcal{D}(U)$ will
denote the set of smooth functions defined in $\mathbb{S}$ whose support is
contained in $U.$

If $F\in\mathcal{S}^{\prime}(\mathbb{R}^{d})$ and $\psi$ is a smooth function
defined in $\mathbb{S}$ then we can define the distribution $f\left(
r\right)  =\left\langle F\left(  r\mathbf{\omega}\right)  ,\psi\left(
\mathbf{\omega}\right)  \right\rangle _{\mathbf{\omega}}$ of the space
$\mathcal{S}^{\prime}\left(  0,\infty\right)  $ by
\begin{equation*}
\left\langle f\left(  r\right)  ,\rho\left(  r\right)  \right\rangle
_{r}=\left\langle F_{n}\left(  \mathbf{x}\right)  ,\left\vert \mathbf{x}
\right\vert ^{1-d}\rho\left(  \left\vert \mathbf{x}\right\vert \right)
\psi\left(  \mathbf{x/}\left\vert \mathbf{x}\right\vert \right)  \right\rangle
\,,\ \ \ \ \ \rho\in\mathcal{S}\left(  0,\infty\right)  \,. 
\end{equation*}
In general this equation cannot be applied if $\rho\in\mathcal{S}_{+},$ so
that $f$ does not belong to $\mathcal{S}_{+}^{\prime},$ but there are always
extensions\footnote{Interestingly \cite{estrada03}, there are no
\emph{continuous} extension operators from $\mathcal{S}^{\prime}\left(
0,\infty\right)  $ to $\mathcal{S}_{+}^{\prime}.$} of $f$ to $\mathcal{S}
_{+}^{\prime}.$

\begin{proposition}
\label{Prop. SV 1}Let $\{F_{n}\}_{n=0}^{\infty}$ be a sequence of
distributions of the space $\mathcal{S}^{\prime}(\mathbb{R}^{d}).$ Suppose
there exists a smooth function $\psi\in\mathcal{D}(V\cap\mathbb{S})$ such that
the generalized functions of one variable $f_{n}\left(  r\right)
=\left\langle F_{n}\left(  r\mathbf{\omega}\right)  ,\psi\left(
\mathbf{\omega}\right)  \right\rangle _{\mathbf{\omega}}$ have
extensions\footnote{These distributions belong to the spaces $\mathcal{R}_{d}$
introduced in \cite{Grafakos-Teschl}.} in $\mathcal{S}_{+}^{\prime}$\ that
satisfy \eqref{P1} and \eqref{P2}. Then for any arbitrary sequence of vectors
$\{\mathbf{a}_{n}\}_{n=0}^{\infty}$ in a Fr\'{e}chet space $\mathcal{X}$ the moment problem (\ref{SV.1}) has solutions in
$\mathcal{S}(V,\mathcal{X}).$
\end{proposition}

\begin{proof}
Indeed, the moment problem $\left\langle f_{n}\left(  r\right)  ,\varphi
\left(  r\right)  \right\rangle =\mathbf{a}_{n}$ has solutions $\varphi\in
\mathcal{S}(\left(  0,\infty\right),\mathcal{X})  .$ If we put
\begin{equation*}
\phi\left(  \mathbf{x}\right)  =\left\vert \mathbf{x}\right\vert ^{1-d}
\varphi\left(  \left\vert \mathbf{x}\right\vert \right)  \psi\left(
\mathbf{x/}\left\vert \mathbf{x}\right\vert \right)  \,, 
\end{equation*}
then $\phi\in\mathcal{S}(V,\mathcal{X})$ and it satisfies the moment problem
(\ref{SV.1}).
\end{proof}

Proposition \ref{Prop. SV 1} implies that if the $F_{n}$'s are radial
distributions, that is, they depend only on $\left\vert \mathbf{x}\right\vert
,$ $F_{n}\left(  \mathbf{x}\right)  =f_{n}\left(  \left\vert \mathbf{x}
\right\vert \right)  $ and the distributions\footnote{The distributions of one
variable $f_{n}$ are not unique \cite{Estrada14}.} $f_{n}$ satisfy \eqref{P1} and \eqref{P2}, then (\ref{SV.1}) has solutions in $\mathcal{S}(V,\mathcal{X})$ for \emph{any} cone
$V$: we may just take $\psi$ as any element of $\mathcal{S}(V\cap\mathbb{S})$
whose integral over $V\cap\mathbb{S}$ does not vanish. Actually we can improve
this result:

\begin{corollary}
\label{Cor SV 1} Let $\mathcal{X}$ be a Fr\'{e}chet space and let
$\{\mathbf{a}_{k}\}$ be an arbitrary sequence of elements of $\mathcal{X}.$
 Let $\{f_{n}\}_{n=0}^{\infty}$ be a sequence of distributions of $\mathcal{S}_{+}'$ for which \eqref{P1} and \eqref{P2} hold and let $\{g_{n}\}_{n=0}^{\infty}$ be a distribution sequence in $\mathcal{D}'(\mathbb{S})$. If $\{F_{n}\}_{n=0}^{\infty}$ is any sequence of distributions in $\mathcal{S}'(\mathbb{R}^{d})$ such that 
$$
\langle F_{n}(\mathbf{x}), \phi(\mathbf{x})  \rangle= \langle f_{n}(r)\otimes g_{n}(\omega) , r^{d-1}\phi(r \omega)\rangle \quad \mbox{for each } \phi\in \mathcal{S}(V),
$$
then (\ref{SV.1}) is solvable in $\mathcal{S}(V,\mathcal{X})$ for any open cone
$V$ provided that there is an open set $U$ of $\mathbb{S}$ such that $\overline{U}\subset V\cap \mathbb{S}$ and $g_{n}\neq 0$ on $U$. 
\end{corollary}
\begin{proof}
In fact, using Proposition \ref{Prop. SV 1}, it suffices to check that there is a smooth function $\psi$ with $\operatorname*{supp} \psi \subseteq \overline{U}$ such that $\langle g_{n}(\omega),\psi(\omega)\rangle\neq 0$ for all $n\in \mathbb{N}$. To show this, consider the Fr\'{e}chet space $\mathcal{Y}=\{\psi\in C^{\infty}(\mathbb{S}):\: \operatorname*{supp} \psi \subseteq \overline{U}\}$ and the sequence of closed subspaces $\mathcal{Y}_{n}=\operatorname{ker}_{\mathcal{Y}} g_{n}=\{\psi\in \mathcal{Y}:\: \langle g_n,\psi \rangle=0\}$. Since each $g_n$ is non-identically zero on $U$, we have that $\mathcal{Y}_{n}\neq \mathcal{Y}$ is of the first category. The Baire theorem implies that $\mathcal{Y}\neq \bigcup_{n=0}^{\infty}\mathcal{Y}_{n}$.
\end{proof}


\begin{thebibliography}{99}                                                                                               


\bibitem {boas1939}R.~P.~Boas, \emph{The Stieltjes moment problem for
functions of bounded variation,} Bull. Amer. Math. Soc. \textbf{45} (1939), 399--404.

\bibitem{campos} J.~Campos Ferreira, \emph{Introduction to the theory of distributions}, Longman, Harlow, 1997.

\bibitem {C-C-K2003}J.~Chung, S.~-Y.~Chung, D.~Kim, \emph{Every Stieltjes
moment problem has a solution in Gel'fand-Shilov spaces,} J. Math. Soc. Japan
\textbf{55} (2003), 909--913.


\bibitem {duran1989}A.~J.~Dur\'{a}n, \emph{The Stieltjes moments problem for
rapidly decreasing functions,} Proc. Amer. Math. Soc. \textbf{107} (1989), 731--741.

\bibitem{duran1992} A.~J.~Dur\'{a}n, \emph{The Stieltjes moment problem with complex exponents,} Math. Nachr. \textbf{158} (1992), 175--194.

\bibitem {duran-estrada1994}A.~L.~Dur\'{a}n, R.~Estrada, \emph{Strong moment
problems for rapidly decreasing smooth functions,} Proc. Amer. Math. Soc.
\textbf{120} (1994), 529--534.

\bibitem {estrada1998}R.~Estrada, \emph{Vector moment problems for rapidly
decreasing smooth functions of several variables,} Proc. Amer. Math. Soc.
\textbf{126} (1998), 761--768.

\bibitem {estradaCesaro}R.~Estrada, \emph{The Ces\`{a}ro behaviour of
distributions,} R. Soc. Lond. Proc. Ser. A Math. Phys. Eng. Sci. \textbf{454}
(1998), 2425--2443.

\bibitem {estrada03}R.~Estrada, \emph{The non-existence of regularization
operators,} J. Math. Anal. Appl. \textbf{286 }(2003), 1--10.

\bibitem {Estrada14}R.~Estrada, \emph{On radial functions and distributions
and their Fourier transforms,} J. Fourier Anal. Appl. \textbf{20 }(2014), 301--320.

\bibitem {estrada-kanwal2002}R.~Estrada, R.~P.~Kanwal, \emph{A distributional
approach to asymptotics. Theory and applications,} Second edition,
Birkh\"{a}user Boston, Inc., Boston, MA, 2002.

\bibitem {estrada-vindasGIntegral}R.~Estrada, J.~Vindas, \emph{A general
integral,} Dissertationes Math. \textbf{483} (2012), 1--49.

\bibitem {estrada-vindasTauber2nd}R.~Estrada, J.~Vindas, \emph{On Tauber's
second Tauberian theorem,} Tohoku Math. J. \textbf{64} (2012), 539--560.

\bibitem {galindo-sanz2001}F.~Galindo, F.~L\'{o}pez, J.~Sanz, \emph{A
generalized moment problem for rapidly decreasing smooth vector functions of
several variables,} J. Math. Anal. Appl. \textbf{263} (2001), 655--665.

\bibitem {Grafakos-Teschl}L.~Grafakos, G.~Teschl, \emph{On Fourier transforms
of radial functions and distributions,} J. Fourier Anal. Appl. \textbf{19
}(2013), 167--179.

\bibitem {kotheII}G.~K\"{o}the, \emph{Topological vector spaces. II,}
Springer-Verlag, New York-Berlin, 1979.

\bibitem {lastra-sanz2008}A.~Lastra, J.~Sanz, \emph{Linear continuous
operators for the Stieltjes moment problem in Gelfand-Shilov spaces,} J. Math.
Anal. Appl. \textbf{340} (2008), 968--981.

\bibitem {lastra-sanz2009}A.~Lastra, J.~Sanz, \emph{Stieltjes moment problem
in general Gelfand-Shilov spaces,} Studia Math. \textbf{192} (2009), 111--128.

\bibitem {lojasiewicz}S.~\L ojasiewicz, \textit{ Sur la valuer et la limite
d'une distribution en un point}, Studia Math. \textbf{16} (1957), 1--36.

\bibitem {P-S-V}S.~Pilipovi\'{c}, B.~Stankovi\'{c}, J.~Vindas,
\emph{Asymptotic behavior of generalized functions,} Series on Analysis,
Applications and Computations, 5, World Scientific Publishing Co.,
Hackensack, NJ, 2012.

\bibitem {polya1938}G.~P\'{o}lya, \emph{Sur l'indetermination d'un
th\'{e}or\`{e}me voisin du probleme des moments,} C.R. Acad. Sci. Paris
\textbf{207} (1938), 708--711.

\bibitem {rr}A. P.~Robertson, W.~Robertson, \emph{Topological vector spaces,}
Cambridge University Press, London-New York, 1973.

\bibitem {silva1955}J.~Sebasti\~{a}o~e~Silva, \emph{Su certe classi di spazi
localmente convessi importanti per le applicazioni,} Rend. Mat. e Appl. (5)
\textbf{14} (1955), 388--410.

\bibitem {shohat-tamarkin}J.~A.~Shohat, J.~D.~Tamarkin, \emph{The problem of
moments,} American Mathematical Society Mathematical Surveys, vol. I, AMS, New
York, 1943.

\bibitem {Stieltjes}T.~J.~Stieltjes, \emph{Recherches sur les fractions continues,}
Reprint of the 1894 original, Ann. Fac. Sci. Toulouse Math. (6) \textbf{4} (1995), pp.
1--35 \& 36--75.

\bibitem {treves}F.~Tr\`{e}ves, \textit{Topological vector spaces,
distributions and kernel}, Academic Press, New York, 1967.

\bibitem {vindas-estrada2010}J.~Vindas, R.~Estrada, \textit{On the support of
tempered distributions,} Proc. Edin Math. Soc. \textbf{53} (2010), 255--270.

\bibitem {vladimirovbook}V.~S.~Vladimirov, \textit{Methods of the theory of
generalized functions,} Analytical Methods and Special Functions, 6, Taylor \&
Francis, London, 2002.
\end{thebibliography}
\end{document}